\documentclass[11pt]{article}

\usepackage{macros}
\usepackage{cite}
\usepackage{hyperref}
\usepackage[T1]{fontenc}
\usepackage{cleveref}
\usepackage{a4wide}
\usepackage{lipsum}
\usepackage{tikz}
\usepackage{lmodern}

\hypersetup{
  colorlinks = true,
  urlcolor   = red,
  linkcolor  = blue,
  citecolor  = red
}

\makeatletter

\@addtoreset{theorem}{section}
\@addtoreset{proposition}{section}
\@addtoreset{corollary}{section}
\@addtoreset{example}{section}
\@addtoreset{remark}{section}
\@addtoreset{lemma}{section}
\@addtoreset{definition}{section}
\@addtoreset{question}{section}
\@addtoreset{conjecture}{section}

\def\@seccntformat#1{\csname the#1\endcsname\quad}
\makeatother

\begin{document}
\setlength{\baselineskip}{1\baselineskip}

\title{\Huge{On $\Gr{H}$-Intersecting Graph Families\\ and Counting of Homomorphisms}}

\author{Igal Sason
\thanks{
Igal Sason is with the Viterbi Faculty of Electrical and Computer Engineering and also the Department of Mathematics
(a secondary affiliation) at the Technion --- Israel Institute of Technology, Haifa 3200003, Israel. Email: eeigal@technion.ac.il. \\[0.1cm]
{\bf{Citation}}: I. Sason, ``On H-intersecting graph families and counting of homomorphisms,''
{\em AIMS Mathematics}, vol.~10, no.~3, paper~290, pp.~6355--6378, March 2025. \url{https://doi.org/10.3934/math.2025290} \\[0.1cm]
This postprint incorporates the discussion in Remark~\ref{remark: Comparison to Sidorenko's LB}, along with several additional related references.}}

\maketitle

\thispagestyle{empty}
\setcounter{page}{1}

{\footnotesize \noindent {\bf Abstract.}
This work derives an upper bound on the maximum cardinality of a family of graphs on a fixed number of vertices, in which
the intersection of every two graphs in that family contains a subgraph that is isomorphic to a specified graph $\Gr{H}$.
Such families are referred to as $\Gr{H}$-intersecting graph families. The bound is derived using the combinatorial version
of Shearer's lemma, and it forms a nontrivial extension of the bound derived by Chung, Graham, Frankl, and Shearer (1986),
where $\Gr{H}$ is specialized to a triangle. The derived bound is expressed in terms of the chromatic number of $\Gr{H}$,
while a relaxed version, formulated using the Lov\'{a}sz $\vartheta$-function of the complement of $\Gr{H}$, offers reduced
computational complexity. Additionally, a probabilistic version of Shearer’s lemma, combined with properties of Shannon
entropy, are employed to establish bounds related to the enumeration of graph homomorphisms, providing further insights into
the interplay between combinatorial structures and information-theoretic principles.

\vspace*{0.1cm}

\noindent {\bf 2020 Mathematics Subject Classification (MSC).} 05C30, 05C60, 05C80, 94A15.}

\section{Introduction}
\label{section: introduction}

An $\Gr{H}$-intersecting family of graphs is a collection of finite, undirected, and simple graphs (i.e., graphs with no self-loops or
parallel edges) on a fixed number of vertices, in which the intersection of every two graphs in the family contains a subgraph isomorphic to
$\Gr{H}$. For instance, if $\Gr{H}$ is an edge or a triangle, then every pair of graphs in the family shares at least one edge or triangle,
respectively. These intersecting families of graphs play a central role in extremal graph theory, where determining their maximum possible
size remains a longstanding challenge. Different choices of $\Gr{H}$ lead to distinct combinatorial problems and structural constraints.

A pivotal conjecture, proposed in 1976 by Simonovits and S\'{o}s, concerned the maximum size of triangle-intersecting graph families---those
in which the intersection of any two graphs contains a triangle. They conjectured that the largest size of such a family
is obtained by the family of all graphs on $n$ vertices that contain a fixed triangle, leading to the conjectured largest size
of $2^{\binom{n}{2}-3}$. Furthermore, a trivial upper bound on the size of a triangle-intersecting graph family on $n$ vertices
is~4 times larger than this conjectured value. This holds since a graph and its complement cannot both belong to an edge-intersecting
family, let alone a triangle-intersecting family.
The foundational work in \cite{SimonovitsS76,SimonovitsS78,SimonovitsS80} explores intersection theorems for graph families whose
shared subgraphs are cycles or paths.

The first major progress on the conjecture by Simonovits and S\'{o}s was made in 1986 by Chung, Graham, Frankl, and Shearer \cite{ChungGFS86},
who utilized Shearer’s inequality to establish a non-trivial upper bound on the largest possible cardinality of a family of triangle-intersecting
graphs with a fixed number of vertices. This bound lay between the trivial and conjectured bounds (or, more formally, it
was equal to twice the conjectured bound, which is also the geometric mean of the conjectured and trivial bounds).

The conjecture by Simonovits and S\'{o}s was ultimately resolved in 2012 by Ellis, Filmus, and Friedgut \cite{EllisFF12}, who proved that
the largest triangle-intersecting family comprises all graphs containing a fixed triangle. Building on the Fourier analytic methods
of Boolean functions, used to prove this conjecture in \cite{EllisFF12} (see also Section~4 of \cite{Ellis22}), a recent work by
Berger and Zhao \cite{BergerZ23} extended the investigation to $\CoG{4}$-intersecting graph families, addressing analogous questions
for graph families where every pair of graphs intersects in a complete subgraph of size four.
Additionally, Keller and Lifshitz \cite{KellerL19} constructed, for every graph $\Gr{H}$ and for every $p \in (\tfrac12, 1)$, an
$\Gr{H}$-intersecting family of graphs $\set{G}$ on $n$ vertices such that a random graph $\Gr{G} \sim \Gr{G}(n,p)$ belongs to $\set{G}$
with probability tending to~1 exponentially fast in $n^2$. Here, $\Gr{G}(n,p)$ denotes the (binomial) Erd\v{o}s-R\'{e}nyi random graph,
in which every edge of $\CoG{n}$ (the complete graph on $n$ vertices) is included independently with probability $p$.
These contributions highlight the interplay between combinatorial, probabilistic, and algebraic methods in the analysis of intersecting graph families.

The interplay between Shannon entropy and extremal combinatorics has significantly enhanced the understanding of the structural and
quantitative properties of combinatorial objects through information-theoretic methods. Entropy serves as a versatile and powerful
tool to derive concise, often elegant proofs of classical results in extremal combinatorics (see, e.g., Chapter~37 of \cite{AignerZ18},
Chapter~22 of \cite{Jukna11}, and \cite{Galvin14,Pippenger77,Pippenger99,Radhakrishnan97,Radhakrishnan01}).
Notable examples include Radhakrishnan's entropy-based proof of Bregman's theorem on matrix permanents \cite{Radhakrishnan97}
and the application of Shearer's lemma to upper-bound the size of the largest triangle-intersecting graph families with a fixed
number of vertices \cite{ChungGFS86}. Beyond this specific context, Shearer's inequalities have found extensive applications across
diverse areas, including finite geometry, graph theory, the analysis of Boolean functions, and large deviations
(see \cite{ChungGFS86, Friedgut04, GavinskyLSS14, Kahn01, Kahn02, MadimanT_IT10, Radhakrishnan01, Sason21, Sason22}).
A recent talk by the author addresses Shearer's inequalities and their applications in combinatorics \cite{Sason_HIM2024}.

The first part of this paper relies on the combinatorial version of Shearer's inequalities in \cite{ChungGFS86} to derive a new upper
bound on the cardinality of families of $\Gr{H}$-intersecting graphs with a fixed number of vertices. The bound represents a nontrivial
extension of the bound in \cite{ChungGFS86}, where $\Gr{H}$ is specialized to a triangle.
The derived bound is expressed in terms of the chromatic number of $\Gr{H}$, while a relaxed version, formulated using the Lov\'{a}sz
$\vartheta$-function of the complement of $\Gr{H}$, reduces computational complexity. The relaxed bound is further explored in the
case where $\Gr{H}$ is a regular graph, particularly when it is strongly regular.

Graph homomorphisms serve as a versatile framework to understand graph mappings, which facilitate studies of structural
properties, colorings, and symmetries. The applications of graph homomorphisms span various fields, including statistical physics, where they model
spin systems \cite{BrightwellW99}, and computational complexity, where they underpin constraint satisfaction problems \cite{HellN04}.
Recent research has yielded profound insights into counting graph homomorphisms, a problem with deep theoretical and practical relevance
(see \cite{Borgs06, CsikvarRS22, FriedgutK98, Galvin14, Lovasz12, WangTL23, Sidorenko93, Simonovits84, Zhao23,
ConlonFS10, ConlonFS10b, ConlonKLL18, Szegedy15a, Szegedy15b, SahSSZ20}). The second part of this paper relies on Shearer's inequalities
and properties of Shannon entropy to derive bounds on the number of homomorphisms.

The paper is structured as follows: Section~\ref{section: Preliminaries} presents essential preliminary material,
including three versions of Shearer’s inequalities.
In Section~\ref{section: entropy bounds - Intersecting Families of Graphs}, the combinatorial version of Shearer’s
lemma is employed for upper bounding the size of $\Gr{H}$-intersecting families of graphs.
Section~\ref{section: entropy bounds - Number of Graph Homomorphisms} focuses on entropy-based proofs, also
incorporating a probabilistic version of Shearer’s lemma, to derive bounds on the number of graph homomorphisms.

\section{Preliminaries}
\label{section: Preliminaries}

\subsection{Shearer's inequalities}
\label{subsection: Shearer's inequalities}

The following subsection introduces three versions of Shearer's inequalities that are useful in the analysis presented
in this paper. The first version serves as a foundation for proving the other two, which are directly applied in this
work. Familiarity with Shannon entropy and its basic properties is assumed, following standard notation (see, e.g.,
Chapter~3 of \cite{CoverT06}).

\begin{proposition}[Shearer's Lemma]
\label{proposition: Shearer's Lemma}
{\em Let
\begin{itemize}
\item $n, m, k \in \naturals$,
\item $X_1, \ldots, X_n$ be discrete random variables,
\item $\OneTo{n} \triangleq \{1, \ldots, n\}$,
\item $\set{S}_1, \ldots, \set{S}_m \subseteq \OneTo{n}$ be subsets such that each
$i \in \OneTo{n}$ belongs to at least $k \geq 1$ of these subsets,
\item $X^n \eqdef (X_1, \ldots, X_n)$, and $X_{\set{S}_j} \triangleq (X_i)_{i \in \set{S}_j}$
for all $j \in \OneTo{m}$.
\end{itemize}
Then,
\begin{align}
\label{eq: Shearer's lemma}
k \Ent{X^n} \leq \sum_{j=1}^m \Ent{X_{\set{S}_j}}.
\end{align}}
\end{proposition}

\begin{proof}
By assumption, $d(i) \geq k$ for all $i \in \OneTo{n}$, where
\begin{align}
\label{eq: d}
d(i) \eqdef \bigcard{\bigl\{j \in \OneTo{m}: \, i \in \set{S}_j \bigr\}}.
\end{align}
Let $\set{S} = \{i_1, \ldots, i_\ell\}$, $1 \leq i_1 < \ldots < i_\ell \leq n$, which implies that
$\card{\set{S}} = \ell, \, \set{S} \subseteq \OneTo{n}$. Further, let $X_{\set{S}} \eqdef (X_{i_1}, \ldots, X_{i_\ell})$.
By the chain rule and the fact that conditioning reduces entropy,
\begin{align}
\Ent{X_{\set{S}}} &= \Ent{X_{i_1}} + \EntCond{X_{i_2}}{X_{i_1}} + \ldots + \EntCond{X_{i_\ell}}{X_{i_1}, \ldots, X_{i_{\ell-1}}} \nonumber \\[0.1cm]
&\geq \underset{i \in \set{S}}{\sum} \EntCond{X_i}{X_1, \ldots, X_{i-1}} \nonumber \\
&= \overset{n}{\underset{i=1}{\sum}} \, \Bigl\{ \I{i \in \set{S}} \, \EntCond{X_i}{X_1, \ldots, X_{i-1}} \Bigr\}, \label{eq10: 22.09.2024}
\end{align}
where $\I{i \in \set{S}}$ on the right-hand side of \eqref{eq10: 22.09.2024} denotes the indicator function of the event $\{i \in \set{S}\}$,
meaning that it is equal to~1 if $i \in \set{S}$ and it is zero otherwise. Consequently, we get
\begin{align}
\sum_{j=1}^m \Ent{X_{\set{S}_j}} &\geq \overset{m}{\underset{j=1}{\sum}} \, \overset{n}{\underset{i=1}{\sum}}
\, \Bigl\{ \I{i \in \set{S}_j} \, \EntCond{X_i}{X_1, \ldots, X_{i-1}} \Bigr\} \label{eq1: 07.04.2025} \\
&= \overset{n}{\underset{i=1}{\sum}} \, \Biggl\{ \overset{m}{\underset{j=1}{\sum}} \, \I{i \in \set{S}_j}
\, \EntCond{X_i}{X_1, \ldots, X_{i-1}} \Biggr\} \label{eq2: 07.04.2025} \\
&= \overset{n}{\underset{i=1}{\sum}} \, \Bigl\{ d(i) \, \EntCond{X_i}{X_1, \ldots, X_{i-1}} \Bigr\} \label{eq3: 07.04.2025} \\
&\geq k \, \overset{n}{\underset{i=1}{\sum}} \, \EntCond{X_i}{X_1, \ldots, X_{i-1}} \label{eq: 22.09.2024} \\[0.1cm]
&=k \, \Ent{X^n},  \label{eq: 03.03.2025}
\end{align}
where inequality \eqref{eq1: 07.04.2025} holds by applying \eqref{eq10: 22.09.2024} to get a lower bound on $\Ent{X_{\set{S}_j}}$
for each $j \in \OneTo{m}$; equality \eqref{eq2: 07.04.2025} holds by swapping the order of summation; equality \eqref{eq3: 07.04.2025}
holds by \eqref{eq: d}; inequality \eqref{eq: 22.09.2024} holds due to the nonnegativity of the conditional entropies of discrete
random variables, and since (by assumption) $d(i) \geq k$ for all $i \in \OneTo{n}$, and equality \eqref{eq: 03.03.2025}
holds by the chain rule of Shannon entropy.
\end{proof}

\begin{remark}
\label{remark: Shearer's Lemma - exactly k elements}
{\em If every element $i \in \OneTo{n}$ belongs to {\em exactly} $k$ of the subsets $\set{S}_1, \ldots, \set{S}_m$,
then Shearer's lemma also applies to continuous random variables $X_1, \ldots, X_n$, with entropy replaced by the
differential entropy. This holds since \eqref{eq: 22.09.2024} is satisfied with equality under the condition
that $d(i)=k$ for all $i \in \OneTo{n}$, independently of the nonnegativity of the conditional entropies
$\EntCond{X_i}{X_1, \ldots, X_{i-1}}$ for each $i \in \OneTo{n}$.}
\end{remark}

\begin{example}[Subadditivity of Shannon Entropy]
\label{example: subadditivity}
{\em Let $X_1, \ldots, X_n$ be discrete random variables, $m=n$, and $\set{S}_i = \{i\}$ (singletons) for all $i \in \OneTo{n}$.
Every element $i \in \OneTo{n}$ then belongs to a single set among $\set{S}_1, \ldots, \set{S}_n$
(i.e., $k=1$ in Proposition~\ref{proposition: Shearer's Lemma}). By Shearer's lemma,
\begin{align}
\label{eq: subadditivity}
\Ent{X^n} \leq \sum_{j=1}^n \Ent{X_j},
\end{align}
which expresses the subadditivity property of Shannon entropy for discrete random variables. This
also holds for continuous random variables, where the entropy is replaced by differential entropy
since every element $i \in \OneTo{n}$ is contained in exactly one subset (see Remark~\ref{remark: Shearer's Lemma - exactly k elements}).
Consequently, Shearer's lemma yields the subadditivity property of Shannon entropy for discrete and continuous random variables.}
\end{example}

\begin{example}[Han's Inequality, \cite{Han78}]
{\em Let $X_1, \ldots, X_n$ be discrete random variables.
For every $\ell \in \OneTo{n}$, let $\set{S}_\ell = \OneTo{n} \setminus \{\ell\}$. By Shearer's lemma
(Proposition~\ref{proposition: Shearer's Lemma}) applied to these $n$ subsets of $\OneTo{n}$, since
every element $i \in \OneTo{n}$ is contained in exactly $k=n-1$ of these subsets,
\begin{align}
\label{eq1: Han's inequality}
(n-1) \Ent{X^n} \leq \sum_{\ell=1}^n \Ent{X_1, \ldots, X_{\ell-1}, X_{\ell+1}, \ldots, X_n} \leq n \Ent{X^n}.
\end{align}
An equivalent form of \eqref{eq1: Han's inequality} is given by
\begin{align}
\label{eq2: Han's inequality}
0 \leq \sum_{\ell=1}^n \Bigl\{ \Ent{X^n} - \Ent{X_1, \ldots, X_{\ell-1}, X_{\ell+1}, \ldots , X_n} \Bigr\} \leq \Ent{X^n}.
\end{align}
The equivalent forms in \eqref{eq1: Han's inequality} and \eqref{eq2: Han's inequality} are known as Han's inequality.
Note that, by Remark~\ref{remark: Shearer's Lemma - exactly k elements}, the left-hand side inequality of
\eqref{eq1: Han's inequality} and, equivalently, the right-hand side inequality of \eqref{eq2: Han's inequality} remain
valid for continuous random variables as well.}
\end{example}

In the combinatorial version of Shearer's lemma \cite{ChungGFS86}, next given, the concept of entropy is hidden.
\begin{proposition}[Combinatorial Version of Shearer's Lemma]
\label{proposition: Combinatorial Shearer's Lemma}
{\em Consider the following setting:
\begin{itemize}
\item Let $\mathscr{F}$ be a finite multiset of subsets of $\OneTo{n}$ (allowing repetitions of some subsets),
where each element $i \in \OneTo{n}$ is included in at least $k \geq 1$ sets of $\mathscr{F}$.
\item Let $\mathscr{M}$ be a set of subsets of $\OneTo{n}$.
\item For every set $\set{S} \in \mathscr{F}$, let the trace of $\mathscr{M}$ on $\set{S}$, denoted by
$\text{trace}_{\set{S}}(\mathscr{M})$, be the set of all possible intersections of elements of $\mathscr{M}$
with $\set{S}$, i.e.,
\begin{align}
\label{eq1: 01.12.2024}
\text{trace}_{\set{S}}(\mathscr{M}) \triangleq \bigl\{ \set{A} \cap \set{S}: \, \set{A} \in \mathscr{M} \bigr\}, \quad \forall \, \set{S} \in \mathscr{F}.
\end{align}
\end{itemize}
Then,
\begin{align}
\label{eq2: 01.12.2024}
\card{\mathscr{M}} \leq \prod_{\set{S} \in \mathscr{F}} \bigl| \text{trace}_{\set{S}}(\mathscr{M}) \bigr|^{\frac1k}.
\end{align}}
\end{proposition}

\begin{proof}
\noindent

\begin{itemize}
\item Let $\set{X} \subseteq \OneTo{n}$ be a set that is selected uniformly at random from $\mathscr{M}$.
\item Represent $\set{X}$ by the binary random vector $X^n = (X_1, \ldots, X_n)$, where $X_i=\I{i \in \set{X}}$
for all $i \in \OneTo{n}$, so $X_i=1$ if $i \in \set{X}$ and $X_i=0$ otherwise.
\item For $\set{S} \in \mathscr{F}$, let $X_{\set{S}} = (X_i)_{i \in \set{S}}$. By the maximal entropy
theorem, which states that the entropy of a discrete random variable (or vector) is upper-bounded by the logarithm
of the cardinality of its support, with equality if and only if the variable is uniformly distributed over its support,
we get
\begin{align}
\label{eq3: 01.12.2024}
\Ent{X_{\set{S}}} \leq \log \, \bigl| \text{trace}_{\set{S}}(\mathscr{M}) \bigr|.
\end{align}
\item By the assumption that every element $i \in \OneTo{n}$ is included in at least $k \geq 1$
sets of $\mathscr{F}$, it follows from combining Shearer's lemma (Proposition~\ref{proposition: Shearer's Lemma})
and \eqref{eq3: 01.12.2024} that
\begin{align}
k \, \Ent{X^n} & \leq \sum_{\set{S} \in \mathscr{F}} \Ent{X_{\set{S}}} \nonumber \\
& \leq \sum_{\set{S} \in \mathscr{F}} \log \, \bigl| \text{trace}_{\set{S}}(\mathscr{M}) \bigr|.   \label{eq4: 01.12.2024}
\end{align}
\item The equality $\Ent{X^n} = \log \card{\mathscr{M}}$ holds since $X^n$ is in one-to-one correspondence with $\set{X}$,
which is uniformly selected at random from $\mathscr{M}$. Combining this with \eqref{eq4: 01.12.2024} gives
\begin{align}
\label{eq5: 01.12.2024}
\log \card{\mathscr{M}} \leq \frac1k \sum_{\set{S} \in \mathscr{F}} \log \, \bigl| \text{trace}_{\set{S}}(\mathscr{M}) \bigr|,
\end{align}
and exponentiating both sides of \eqref{eq5: 01.12.2024} gives \eqref{eq2: 01.12.2024}.
\end{itemize}
\end{proof}

The following is the probabilistic entropy-based formulation of Shearer's lemma, which is also applied in this paper.

\begin{proposition}[Probabilistic Version of Shearer's Lemma]
\label{proposition: Shearer's Lemma: 2nd version}
{\em Let $X^n$ be a discrete $n$-dimensional random vector, and let $\set{S} \subseteq \OneTo{n}$
be a random subset of $\OneTo{n}$, independent of $X^n$, with an arbitrary
probability mass function $\pmfOf{\set{S}}$. If there exists $\theta > 0$ such that
\begin{align}
\label{eq: Shearer's lemma - condition}
\Prv{i \in \set{S}} \geq \theta, \quad \forall \, i \in \OneTo{n},
\end{align}
then,
\begin{align}
\label{eq: Shearer's Lemma: 2nd version}
\bigExpecwrt{\set{S}}{\Ent{X_{\set{S}}}} \geq \theta \Ent{X^n}.
\end{align}}
\end{proposition}

\begin{proof}
The expectation of the entropy $\Ent{X_{\set{S}}}$ with respect to the random subset $\set{S} \subseteq \OneTo{n}$,
where (by assumption) $\set{S}$ is independent of $X^n$, gives
\begin{align}
\bigExpecwrt{\set{S}}{\Ent{X_{\set{S}}}}
&= \sum_{\set{S} \, \subseteq \OneTo{n}} \pmfOf{\set{S}}(\set{S}) \, \Ent{X_{\set{S}}} \nonumber \\[0.1cm]
&\geq \sum_{\set{S} \, \subseteq \OneTo{n}}
\Biggl\{ \pmfOf{\set{S}}(\set{S}) \; \overset{n}{\underset{i=1}{\sum}}
\Bigl\{ \I{i \in \set{S}} \, \EntCond{X_i}{X_1, \ldots, X_{i-1}} \Bigr\} \Biggr\} \label{eq1:22.12.2024} \\[0.1cm]
&= \overset{n}{\underset{i=1}{\sum}} \Biggl\{ \sum_{\set{S} \, \subseteq \OneTo{n}}
\Bigl\{ \pmfOf{\set{S}}(\set{S}) \; \I{i \in \set{S}} \Bigr\} \, \EntCond{X_i}{X_1, \ldots, X_{i-1}} \Biggr\} \label{eq2:22.12.2024} \\
&= \overset{n}{\underset{i=1}{\sum}} \Prv{i \in \set{S}} \, \EntCond{X_i}{X_1, \ldots, X_{i-1}} \nonumber \\
&\geq \theta \, \overset{n}{\underset{i=1}{\sum}} \EntCond{X_i}{X_1, \ldots, X_{i-1}} \label{eq1: 23.09.2024} \\
&=\theta \Ent{X^n}, \label{eq: chain rule Ent}
\end{align}
where \eqref{eq1:22.12.2024} holds by \eqref{eq10: 22.09.2024} that is valid for every set $\set{S} \subseteq \OneTo{n}$;
\eqref{eq2:22.12.2024} holds by swapping the order of summation, and \eqref{eq1: 23.09.2024} holds by the assumption that
the random variables $\{X_i\}$ are discrete (so, the conditional entropies are nonnegative) and by the condition in
\eqref{eq: Shearer's lemma - condition}. Finally, \eqref{eq: chain rule Ent} holds by the chain rule of Shannon entropy.
\end{proof}

\begin{remark}
\label{remark 2: Shearer's lemma}
{\em Similarly to Remark~\ref{remark: Shearer's Lemma - exactly k elements}, if
$\Prv{i \in \set{S}} = \theta$ for all $i \in \OneTo{n}$, then
inequality \eqref{eq1: 23.09.2024} holds with equality. Consequently,
if the condition in \eqref{eq: Shearer's lemma - condition} is satisfied with equality
for all $i \in \OneTo{n}$, then \eqref{eq: Shearer's Lemma: 2nd version}
extends to continuous random variables, with entropies replaced by differential entropies.}
\end{remark}

\subsection{Intersecting families of graphs}
\label{subsection: Intersecting Families of Graphs}

We start by considering triangle-intersecting families of graphs, which was the problem in extremal combinatorics
addressed in \cite{ChungGFS86}.

\begin{definition}[Triangle-Intersecting Families of Graphs]
\label{definition: intersecting family of graphs}
{\em Let $\set{G}$ be a family of graphs on the vertex set $\OneTo{n}$, with the property that for
every $\Gr{G}_1, \Gr{G}_2 \in \set{G}$, the intersection $\Gr{G}_1 \cap \Gr{G}_2$ contains a triangle
(i.e, there are three vertices $i, j, k \in \OneTo{n}$ such that each of $\{i,j\}$, $\{i,k\}$, $\{j,k\}$
is in the edge sets of both $\Gr{G}_1$ and $\Gr{G}_2$). The family $\set{G}$ is referred to as a
{\em triangle-intersecting} family of graphs on $n$ vertices.}
\end{definition}

\begin{question}{(Simonovits and S\'os, \cite{SimonovitsS78})}
{\em How large can $\set{G}$ (a family of triangle-intersecting graphs) be?}
\end{question}

The family $\set{G}$ can be as large as $2^{\binom{n}{2}-3}$.
To that end, consider the family ${\mathcal G}$ of all graphs on $n$ vertices
that include a particular triangle.
On the other hand, $\card{\set{G}}$ cannot exceed $2^{\binom{n}{2}-1}$.
The latter upper bound holds since, in general, a family of distinct subsets of
a set of size $m$, where any two of these subsets have a nonempty intersection,
can have a cardinality of at most $2^{m-1}$ ($\set{A}$ and $\cset{A}$
cannot be members of this family). The edge sets of the graphs in $\set{G}$
satisfy this property, with $m=\binom{n}{2}$.

\begin{proposition}[Ellis, Filmus, and Friedgut, \cite{EllisFF12}]
\label{proposition: triangle-intersecting families}
{\em The size of a family ${\mathcal G}$ of triangle-intersecting graphs on $n$ vertices satisfies
$\card{\set{G}} \leq 2^{\binom{n}{2}-3}$, and this upper bound is attained by the family of
all graphs with a common vertex set of $n$ vertices, and with a fixed common triangle.}
\end{proposition}
This result was proved by using discrete Fourier analysis to obtain the sharp bound
in Proposition~\ref{proposition: triangle-intersecting families}, as conjectured by
Simonovits and S\'{o}s \cite{SimonovitsS78}.

The first significant progress toward proving the Simonovits–S\'{o}s conjecture came from an
information-theoretic approach \cite{ChungGFS86}.
Using the combinatorial Shearer lemma (Proposition~\ref{proposition: Combinatorial Shearer's Lemma}),
a simple and elegant upper bound on the size of $\set{G}$ was derived in \cite{ChungGFS86}.
That bound is equal to $2^{\binom{n}{2}-2}$, falling short of the Simonovits–S\'{o}s conjecture by
a factor of~2.

\begin{proposition}[Chung, Graham, Frankl, and Shearer, \cite{ChungGFS86}] \label{proposition: Chung et al., 1986}
{\em Let $\set{G}$ be a family of $\CoG{3}$-intersecting graphs on a common vertex set $\OneTo{n}$.
Then, $\card{\set{G}} \leq 2^{\binom{n}{2}-2}$.}
\end{proposition}

We next consider more general intersecting families of graphs.

\begin{definition}[$\Gr{H}$-Intersecting Families of Graphs]  \label{definition: H-intersecting family of graphs}
{\em Let $\mathcal{G}$ be a family of graphs on a common vertex set. Then, it is said that $\mathcal{G}$ is $\Gr{H}$-intersecting
if for every two graphs $\Gr{G}_1, \Gr{G}_2 \in \mathcal{G}$, the graph $\Gr{G}_1 \cap \Gr{G}_2$ contains a subgraph isomorphic
to $\Gr{H}$.}
\end{definition}

In the following, $\CoG{t}$, for $t \in \naturals$, denotes the complete graph on $t$ vertices. This graph consists of $t$ vertices,
with every pair of vertices being adjacent. For example, $\CoG{2}$ represents an edge, while $\CoG{3}$ corresponds to a triangle.
\begin{example}
{\em Let $\Gr{H} = \CoG{t}$, with $t \geq 2$. Then,
$t=2$ means that $\mathcal{G}$ is edge-intersecting (or simply intersecting),
and $t=3$ means that $\mathcal{G}$ is triangle-intersecting.}
\end{example}

\begin{question}{[Problem in Extremal Combinatorics]}
{\em Given $\Gr{H}$ and $n$, what is the maximum size of an $\Gr{H}$-intersecting family of graphs on $n$ labeled vertices?}
\end{question}

\begin{conjecture}{(Ellis, Filmus, and Friedgut, \cite{EllisFF12})}
\label{conjecture: EllisFF12}
{\em Every $\CoG{t}$-intersecting family of graphs on a common vertex set $\OneTo{n}$ has a size of at most $2^{\binom{n}{2}-\binom{t}{2}}$,
with equality for the family of all graphs containing a fixed clique on $t$ vertices.}
\end{conjecture}

\begin{itemize}
\item For $t=2$, it is trivial (since $\CoG{2}$ is an edge).
\item For $t=3$, it was proved by Ellis, Filmus, and Friedgut \cite{EllisFF12}.
\item For $t=4$, it was recently proved by Berger and Zhao \cite{BergerZ23}.
\item For $t \geq 5$, this problem is left open.
\end{itemize}

\subsection{Counting graph homomorphisms}
\label{subsection: Counting Graph Homomorphisms}
In the sequel, let $\V{\Gr{H}}$ and $\E{\Gr{H}}$ denote the vertex and edge sets of a graph $\Gr{H}$, respectively.
Further, let $\Gr{T}$ and $\Gr{G}$ be finite, simple, and undirected graphs, and denote the edge connecting a pair
of adjacent vertices $u,v \in \V{\Gr{H}}$ by an edge $e = \{u,v\} \in \E{\Gr{H}}$.

\begin{definition}[Homomorphism]
\label{definition: homomorphism}
{\em A {\em homomorphism} from $\Gr{T}$ to $\Gr{G}$, denoted by $\Gr{T} \to \Gr{G}$, is a mapping of the vertices
of $\Gr{T}$ to those of $\Gr{G}$, $\sigma \colon \V{\Gr{T}} \to \V{\Gr{G}}$, such that every edge in $\Gr{T}$
is mapped to an edge in $\Gr{G}$:
\begin{align}
\{u,v\} \in \E{\Gr{T}} \implies \{\sigma(u), \sigma(v)\} \in \E{\Gr{G}}.
\end{align}
On the other hand, non-edges in $\Gr{T}$ may be mapped to the same vertex, a non-edge, or an
edge in $\Gr{G}$.}
\end{definition}

\begin{example}
{\em There is a homomorphism from every bipartite graph $\Gr{G}$ to $\CoG{2}$.
Indeed, let $\V{\Gr{G}} = \set{X} \cup \set{Y}$, where $\set{X}$ and $\set{Y}$ are the two
disjoint partite sets. A mapping that maps every vertex in $\set{X}$ to~'0', and every vertex
in $\set{Y}$ to~'1' is a homomorphism $\Gr{G} \to \CoG{2}$ because every edge in $\Gr{G}$ is
mapped to the edge $\{0,1\}$ in $\CoG{2}$. Note that every non-edge in $\set{X}$ or in $\set{Y}$
is mapped to the same vertex in $\CoG{2}$, and every non-edge between two vertices in $\set{X}$
and $\set{Y}$ is mapped to $\{0,1\}$.}
\end{example}

The following connects graph homomorphisms to graph invariants.
Let $\clnum{\Gr{G}}$ and $\chrnum{\Gr{G}}$ denote the clique number and chromatic number, respectively, of
a finite, simple, and undirected graph $\Gr{G}$. That is, $\clnum{\Gr{G}}$ is the maximum number
of vertices in $\Gr{G}$ such that every two of them are adjacent (i.e., these vertices form the vertex set of
a complete subgraph of $\Gr{G}$), and $\chrnum{\Gr{G}}$ is the smallest number of colors required to color the
vertices of $\Gr{G}$ such that no two adjacent vertices share the same color. Then,
\begin{itemize}
\item $\clnum{\Gr{G}}$ is the largest integer $k$ for which a homomorphism $\CoG{k} \to \Gr{G}$
exists. This holds because the image of a complete graph under a homomorphism is a complete graph
of the same size. This is valid because a homomorphism preserves adjacency, and for a complete graph
$\CoG{k}$, all pairs of vertices are adjacent. To preserve this property, the image of $\CoG{k}$
under the homomorphism must also be a complete graph of size $k$.
\item A graph $\Gr{G}$ is $k$-colorable if and only if it has a homomorphism to the complete graph $\CoG{k}$;
this is because $k$-coloring assigns one of $k$ colors to each vertex such that adjacent vertices receive
different colors, which is equivalent to mapping the vertices of $\Gr{G}$ to the $k$ vertices of $\CoG{k}$
in a way that adjacency is preserved. Consequently, it follows by definition that
$\chrnum{\Gr{G}}$ is the smallest integer $k$ for which there exists a homomorphism $\Gr{G} \to \CoG{k}$.
\end{itemize}

Let $\Hom{\Gr{T}}{\Gr{G}}$ denote the set of all the homomorphisms $\Gr{T} \to \Gr{G}$, and let
\begin{align}
\label{eq: number of homomorphisms}
\homcount{\Gr{T}}{\Gr{G}} \triangleq \bigcard{\Hom{\Gr{T}}{\Gr{G}}}
\end{align}
denote the number of these homomorphisms.

The independence number of a graph $\Gr{G}$, denoted by $\indnum{\Gr{G}}$, is the maximum number of vertices in $\Gr{G}$
such that no two of them are adjacent. It can be formulated as an integer linear program, whose relaxation gives rise to
the following graph invariant.
\begin{definition}[Fractional Independence Number]
\label{definition: fractional independence number}
{\em The fractional independence number of a graph $\Gr{G}$, denoted as $\findnum{\Gr{G}}$,
is a fractional relaxation of the independence number $\indnum{\Gr{G}}$. It is defined as
the optimal value of the following linear program:
\begin{itemize}
\item Optimization variables: $x_v$ for every vertex $v \in \V{\Gr{G}}$.
\item Objective: Maximize $\underset{v \in \V{\Gr{G}}}{\sum} x_v$.
\item Constraints: $x_v \geq 0$ for all $v \in \V{\Gr{G}}$, and $\underset{v \in \set{C}}{\sum} x_v \leq 1$
for every clique $\set{C} \subseteq \V{\Gr{G}}$.
\end{itemize}
This relaxation allows fractional values for $x_v$, in contrast to the integer programming formulation for
$\indnum{\Gr{G}}$, where $x_v$ must be binary (either~0 or~1) for all $v \in \V{\Gr{G}}$. Consequently,
$\findnum{\Gr{G}} \geq \indnum{\Gr{G}}$.}
\end{definition}

The following result was obtained by Alon \cite{Alon81} and by Friedgut and Khan \cite{FriedgutK98},
where the latter provides an entropy-based proof that relies on Shearer's lemma with
an extension of the result on the number of homomorphisms for hypergraphs.
\begin{proposition}[Number of Graph Homomorphisms]
\label{proposition: UB on the number of homomorphisms}
{\em Let $\Gr{T}$ and $\Gr{G}$ be finite, simple, and undirected graphs, having no isolated vertices.
Then,
\begin{align}
\label{eq1: 11.09.2024}
\homcount{\Gr{T}}{\Gr{G}} \leq \Bigl( 2 \, \bigcard{\E{\Gr{G}}} \Bigr)^{\findnum{\Gr{T}}}.
\end{align}
Furthermore, the upper bound in \eqref{eq1: 11.09.2024} is essentially tight for a fixed graph $\Gr{T}$
in the sense that there exists a graph $\Gr{G}$ such that
\begin{align}
\label{eq1b: 11.09.2024}
\homcount{\Gr{T}}{\Gr{G}} \geq \left( \frac{\bigcard{\E{\Gr{G}}}}{\bigcard{\E{\Gr{T}}}}\right)^{\findnum{\Gr{T}}},
\end{align}
so, a comparison between the upper and lower bounds on the number of graph homomorphisms in \eqref{eq1: 11.09.2024}
and \eqref{eq1b: 11.09.2024} indicates that $\homcount{\Gr{T}}{\Gr{G}}$ scales like $\bigcard{\E{\Gr{G}}}^{\, \findnum{\Gr{T}}}$.}
\end{proposition}

\section{Intersecting families of graphs}
\label{section: entropy bounds - Intersecting Families of Graphs}

This section derives an upper bound on the maximum cardinality of a family of $\Gr{H}$-intersecting graphs on a fixed
number of vertices, where the intersection of every two graphs in that family contains a subgraph that
is isomorphic to a specified graph $\Gr{H}$. Specifically, the next result generalizes Proposition~\ref{proposition: Chung et al., 1986}
by extending the proof technique of \cite{ChungGFS86} to apply to all families of $\Gr{H}$-intersecting graphs.

\begin{proposition}
\label{proposition: I.S., 2024}
{\em Let $\Gr{H}$ be a nonempty graph, and let $\set{G}$ be a family of $\Gr{H}$-intersecting graphs
on a common vertex set $\OneTo{n}$. Then,
\begin{align}
\label{eq1: 31.01.2025}
\card{\set{G}} \leq 2^{\binom{n}{2}-(\chrnum{\Gr{H}}-1)}.
\end{align}}
\end{proposition}

\begin{proof}
\noindent

\begin{itemize}
\item Identify every graph $\Gr{G} \in \set{G}$ with its edge set $\E{\Gr{G}}$, and let
$\mathscr{M} = \bigl\{ \E{\Gr{G}}: \Gr{G} \in \set{G} \bigr\}$ (recall that all these
graphs have the common vertex set $\OneTo{n}$).
\item Let $\set{U} = \E{\CoG{n}}$. For every $\Gr{G} \in \set{G}$, we have
$\E{\Gr{G}} \subseteq \set{U}$, and $\card{\set{U}} = \binom{n}{2}$.
\item Let $t \triangleq \chrnum{\Gr{H}}$ be the chromatic number of $\Gr{H}$ (or any
graph isomorphic to $\Gr{H}$).
\item For every unordered equipartition or almost-equipartition of $\OneTo{n}$ into $t-1$
disjoint subsets, i.e., $\overset{t-1}{\underset{j=1}{\bigcup}} \set{A}_j = [n]$,
satisfying $\bigl| |\set{A}_i|-|\set{A}_j|\bigr| \leq 1$ and $\set{A}_i \cap \set{A}_j = \es$
for all $1 \leq i < j \leq t-1$, define $\set{U}(\{\set{A}_j\}_{j=1}^{t-1})$ as the subset of
$\mathcal{U}$ consisting of all edges that are entirely contained within some subset $\set{A}_j$
for $j \in \OneTo{t-1}$. In other words, each edge lies entirely within one of the subsets
$\{\set{A}_j\}_{j=1}^{t-1}$, although different edges may belong to different subsets.
\item
We apply the combinatorial Shearer lemma (Proposition~\ref{proposition: Combinatorial Shearer's Lemma}) with
\begin{align}
\label{eq2: 31.01.2025}
\mathscr{F} = \{ \set{U}(\{\set{A}_j\}_{j=1}^{t-1})\},
\end{align}
where $\mathscr{F}$ is obtained by referring in the right-hand side of \eqref{eq2: 31.01.2025}
to all possible choices of $\{\set{A}_j\}_{j=1}^{t-1}$, as described in the previous item.
\item
Let $m=| \, \set{U}(\{\set{A}_j\}_{j=1}^{t-1}) \, |$. Then, $m$ is independent of
$\{\set{A}_j\}_{j=1}^{t-1}$ as described above, since
\begin{align}
\label{eq1:17.12.2024}
m = \left\{
\begin{array}{ll}
(t-1) \, \binom{n/(t-1)}{2} & \mbox{if $(t-1)|n$}, \\[0.2cm]
(t-2) \, \binom{\lfloor n/(t-1) \rfloor}{2} + \binom{\lceil n/(t-1) \rceil}{2} & \mbox{if $(t-1)|(n-1)$,}\\[0.2cm]
\hspace*{3cm} \vdots \\[0.2cm]
\binom{\lfloor n/(t-1) \rfloor}{2} + (t-2) \, \binom{\lceil n/(t-1) \rceil}{2} & \mbox{if $(t-1)|\bigl(n-(t-2)\bigr)$.}
\end{array}
\right.
\end{align}
\item
By \eqref{eq1:17.12.2024} with $t = \chrnum{\Gr{H}}$, it follows that
\begin{align}
\label{eq1b:17.12.2024}
m \leq \frac{1}{\chrnum{\Gr{H}}-1} \, \binom{n}{2}.
\end{align}

\begin{proof}
The graph $\Gr{H}$ is nonempty, so $t = \chrnum{\Gr{H}} \geq 2$. If $(t-1)|n$, then,
\vspace{6pt}
\begin{align}
(t-1) \, \binom{n/(t-1)}{2} &= \frac{n(n-(t-1))}{2(t-1)} \nonumber \\
&\leq \frac1{t-1} \, \binom{n}{2}.   \label{eq1a:30.1.2025}
\end{align}
Otherwise, $(t-1)|(n-j)$ for some integer $j \in \OneTo{t-2}$, so $n = j+r(t-1)$ with $r \in \naturals$.
Consequently, since $j \in \{1, \ldots, t-2\}$, \eqref{eq1:17.12.2024} gives
\vspace{6pt}
\begin{align}
m &= (t-1-j) \, \binom{r}{2} + j \, \binom{r+1}{2} \nonumber \\
&= \frac{r}{2} \, \Bigl[ (t-1-j)(r-1)+j(r+1) \Bigr] \nonumber \\
&= \frac{n-j}{2(t-1)} \, \Bigl[ (t-1-j) \, \Bigl( \frac{n-j}{t-1}-1 \Bigr) + j \, \Bigl(\frac{n-j}{t-1} + 1\Bigr) \Bigr] \nonumber \\
&= \frac{n-j}{2(t-1)^2} \, \Bigl[ (t-1-j)(n-j-t+1) + j(n-j + t-1) \Bigr] \nonumber \\
&= \frac{(n-j) \, \bigl[ n(t-1) + j(t-1) - (t-1)^2 \bigr]}{2(t-1)^2} \nonumber \\
&= \frac{(n-j) \, \bigl(n+j-(t-1)\bigr)}{2(t-1)} \nonumber \\
&\leq \frac{(n-j)(n-1)}{2(t-1)} \nonumber \\
&< \frac1{t-1} \, \binom{n}{2}.   \label{eq1b:30.1.2025}
\end{align}
Combining \eqref{eq1a:30.1.2025} and \eqref{eq1b:30.1.2025} for the cases where $(t-1)|n$
or $(t-1) \hspace*{-0.1cm} \not| \hspace*{0.1cm} n$, respectively, gives \eqref{eq1b:17.12.2024}.
\end{proof}
\item By a symmetry consideration, which follows from the construction of $\mathscr{F}$ in \eqref{eq2: 31.01.2025}
as a set of subsets of $\set{U}$ (due to the fourth item of this proof),
each of the $\binom{n}{2}$ elements in $\set{U}$ (i.e., each edge of the complete graph $\CoG{n}$)
belongs to a fixed number $k$ of elements in $\mathscr{F}$.
\item By a double-counting argument on the edges of the complete graph $\CoG{n}$ (the
set $\set{U}$), since each element of $\mathscr{F}$ is a subset of $\set{U}$ of size $m$ (see
\eqref{eq1:17.12.2024}) and each edge in $\set{U}$ belongs to exactly $k$ elements in $\mathscr{F}$,
the following equality holds:
\begin{equation} \label{eq2-doublecounting}
m \, |\mathscr{F}| = \binom{n}{2} \, k.
\end{equation}
\item Let $\set{S} \in \mathscr{F}$. Observe that ${\rm trace}_{\, \set{S}}(\mathscr{M})$, as
defined in \eqref{eq1: 01.12.2024}, forms an intersecting family of subsets of $\set{S}$. Indeed,
\begin{enumerate}
\item Assign to each vertex in $\OneTo{n}$ the index $j$ of the subset $\set{A}_j$
($1 \leq j \leq \chrnum{\Gr{H}}-1$) in the partition of $\OneTo{n}$ corresponding to $\set{S}$.
Let these assignments be associated with $\chrnum{\Gr{H}}-1$ color classes of the vertices.
\item For any $\Gr{G}, \Gr{G}' \in {\mathcal G}$, the graph $\Gr{G} \cap \Gr{G}'$ contains a subgraph
isomorphic to $\Gr{H}$ (by assumption).
\item By definition, $t = \chrnum{\Gr{H}}$ is the smallest number of colors required to properly
color the vertices of $\Gr{H}$, ensuring that no two adjacent vertices in $\Gr{H}$ share the same color.
Hence, in any coloring of the vertices of any graph isomorphic to $\Gr{H}$ with $t-1$
colors, there must exist at least one edge whose two endpoints are assigned the same color. This implies
that such an edge must belong to some subset $\set{A}_j$ for $j \in \OneTo{\chrnum{\Gr{H}}-1}$, and therefore,
it is contained in $\set{S}$.
\item Let $\Gr{G}, \Gr{G}' \in {\mathcal G}$.
Hence, at least one of the edges in $\E{\Gr{G}} \cap \E{\Gr{G}'}$, which contains the edges of
a graph isomorphic to $\Gr{H}$, belongs to $\set{S}$. By the definition of $\mathscr{M}$ (see the first item),
this means by \eqref{eq1: 01.12.2024} that ${\rm trace}_{\, \set{S}}(\mathscr{M})$ is an intersecting family
of subsets of $\set{S}$ since $$\bigl(\E{\Gr{G}} \cap \set{S} \bigr) \cap \bigl(\E{\Gr{G}'} \cap \set{S} \bigr)
= \bigl(\E{\Gr{G}} \cap \bigl(\E{\Gr{G}'} \bigr) \cap \set{S} \neq \es.$$
\end{enumerate}
Consequently, $\card{\set{S}}=m$ and ${\rm trace}_{\, \set{S}}(\mathscr{M})$ forms an intersecting family
of subsets of $\set{S}$, which yields
\begin{align}
\label{eq1c:17.12.2024}
\bigcard{{\rm trace}_{\, \set{S}} (\mathscr{M})} \leq 2^{m - 1}.
\end{align}
This holds since the total number of subsets of $\set{S}$ is $2^m$,
and since ${\rm trace}_{\, \set{S}}(\mathscr{M})$ forms an intersecting family of these subsets,
it cannot contain any subset along with its complement with respect to $\set{S}$. Consequently,
the cardinality of an intersecting family of subsets of $\set{S}$ is at most $\tfrac12 \cdot 2^m = 2^{m-1}$.
\item
By Proposition~\ref{proposition: Combinatorial Shearer's Lemma} (and the one-to-one correspondence
between $\set{G}$ and~$\mathscr{M}$),
\begin{align}
\card{\mathcal{G}} &= \card{\mathscr{M}} \nonumber \\
& \leq \left(2^{m - 1}\right)^\frac{|\mathscr{F}|}{k} \label{eq1d: 17.12.2024} \\
& = 2^{\binom{n}{2}\left(1-\frac{1}{m}\right)} \label{eq2: 17.12.2024} \\
& \leq 2^{\binom{n}{2}-(\chrnum{\Gr{H}}-1)},  \label{eq3: 17.12.2024}
\end{align}
where \eqref{eq1d: 17.12.2024} relies on \eqref{eq2: 01.12.2024} and \eqref{eq1c:17.12.2024}, equality
\eqref{eq2: 17.12.2024} relies on \eqref{eq2-doublecounting}, and \eqref{eq3: 17.12.2024} holds due to \eqref{eq1b:17.12.2024}.
\end{itemize}
\end{proof}

The family $\set{G}$ of $\Gr{H}$-intersecting graphs on $n$ vertices can be as large as $2^{\binom{n}{2}-\card{\hspace*{-0.03cm} \E{\Gr{H}}}}$.
To that end, consider the family ${\mathcal G}$ of all graphs on $n$ vertices that include $\Gr{H}$ as a subgraph.

Combining this lower bound on $\card{\set{G}}$ with its upper bound in Proposition~\ref{proposition: I.S., 2024} gives
that the largest family $\set{G}$ of $\Gr{H}$-intersecting graphs on $n$ vertices satisfies
\begin{align}
\label{eq: UB and LB}
2^{\binom{n}{2}-\card{\hspace*{-0.03cm} \E{\Gr{H}}}} \leq \card{\set{G}} \leq 2^{\binom{n}{2}-(\chrnum{\Gr{H}}-1)}.
\end{align}

Specialization of Proposition~\ref{proposition: I.S., 2024} to complete subgraphs gives the following.
\begin{corollary}
\label{corollary: complete graphs}
{\em Let $\set{G}$ be a family of $\CoG{t}$-intersecting graphs, with $t \geq 2$, on a common vertex set $\OneTo{n}$.
Then,
\begin{align}
\label{eq1b: 31.01.2025}
\card{\set{G}} \leq 2^{\binom{n}{2}-(t-1)}.
\end{align}}
\end{corollary}
\begin{proof}
$\chrnum{\CoG{t}} = t$ for all $t \in \naturals$, and if $t \geq 2$, then the complete graph $\CoG{t}$ is nonempty.
\end{proof}

\begin{remark}
{\em For $t \geq 3$, the bound in Corollary~\ref{corollary: complete graphs} falls short of the
conjectured result in \cite{EllisFF12}, which states that every $\CoG{t}$-intersecting family
of graphs on a common vertex set $\OneTo{n}$ has a size of at most $2^{\binom{n}{2}-\frac{t(t-1)}{2}}$,
with equality achieved by the family of all graphs containing a fixed clique on $t$ vertices.
Nevertheless, it generalizes Proposition~\ref{proposition: Chung et al., 1986}, which addresses
the special case $\Gr{H}=\CoG{3}$ (triangle-intersecting graphs), and it uniformly improves the
trivial bound of $2^{\binom{n}{2}-1}$ for $\CoG{t}$-intersecting graph families on $n$ vertices
with $t \geq 3$.}
\end{remark}

\begin{remark}
{\em If $\Gr{H}$ is a bipartite graph, then $\chrnum{\Gr{H}} = 2$. In that case, our result for an $\Gr{H}$-intersecting
family of graphs on $n$ vertices is specialized to the trivial bound
\begin{align}
\label{eq1c: 31.01.2025}
\card{\set{G}} \leq 2^{\binom{n}{2}-1}.
\end{align}
This bound is tight for the largest family of $\CoG{2}$-intersecting graphs on $n$ vertices, consisting of all graphs that
contain a fixed edge as a subgraph (since in this case, the upper and lower bounds in \eqref{eq: UB and LB} coincide).}
\end{remark}

The computational complexity of the chromatic number of a graph is in general NP-hard \cite{GareyJ79}. This poses a problem in
calculating the upper bound in Proposition~\ref{proposition: I.S., 2024} on the cardinality of $\Gr{H}$-intersecting
families of graphs on a fixed number of vertices. This bound can be however loosened, expressing it in terms of the
Lov\'{a}sz $\vartheta$-function of the complement graph $\CGr{H}$ (see Corollary~3 of \cite{Lovasz79_IT}).

\begin{corollary}
\label{corollary: Lovasz function}
{\em Let $\Gr{H}$ be a graph, and let $\set{G}$ be a family of $\Gr{H}$-intersecting graphs
on a common vertex set $\OneTo{n}$. Then,
\begin{align}
\label{eq:23.01.2025}
\card{\set{G}} \leq 2^{\binom{n}{2}-(\lceil\vartheta(\CGr{H})\rceil-1)}.
\end{align}}
\end{corollary}
\begin{proof}
The Lov\'{a}sz $\vartheta$-function of the complement graph $\CGr{H}$ satisfies the inequality
\begin{align}
\label{eq: sandwich inequality}
\clnum{\Gr{H}} \leq \vartheta(\CGr{H}) \leq \chrnum{\Gr{H}},
\end{align}
so it is bounded between the clique and chromatic numbers of $\Gr{H}$, which are both NP-hard to compute \cite{GareyJ79}.
Since the chromatic number $\chrnum{\Gr{H}}$ is an integer, we have
\begin{align}
\label{eq1d: 31.01.2025}
\chrnum{\Gr{H}} \geq \lceil\vartheta(\CGr{H}) \rceil.
\end{align}
Combining \eqref{eq1: 31.01.2025} and \eqref{eq1d: 31.01.2025} yields \eqref{eq:23.01.2025}.
\end{proof}

The Lov\'{a}sz $\vartheta$-function of the complement graph $\CGr{H}$, as presented in Corollary~\ref{corollary: Lovasz function},
can be efficiently computed with a precision of $r$ decimal digits, having a computational complexity that is polynomial in
$p \triangleq \bigcard{\V{\Gr{H}}}$ and $r$. It can be computed by solving the following semidefinite programming (SDP) problem~\cite{GrotschelLS81}:
\begin{equation}
\label{eq: SDP problem}
\mbox{\fbox{$
\begin{array}{l}
\text{maximize} \; \; \mathrm{Tr}({\bf{B}} \, {\bf{J}}_p) \\
\text{subject to} \\
\begin{cases}
{\bf{B}} \in {\bf{\set{S}}}_{+}^p, \; \; \mathrm{Tr}({\bf{B}}) = 1, \\
A_{i,j} = 0  \; \Rightarrow \;  B_{i,j} = 0, \quad i,j \in \OneTo{p}, \; i \neq j,
\end{cases}
\end{array}$}}
\end{equation}
where the following notation is used:
\begin{itemize}
\item ${\bf{A}} = {\bf{A}}(\Gr{H})$ is the $p \times p$ adjacency matrix of $\Gr{H}$;
\item ${\bf{J}}_p$ is the all-ones $p \times p$ matrix;
\item ${\bf{\set{S}}}_{+}^p$ is the set of all $p \times p$ positive semidefinite matrices.
\end{itemize}
The reader is referred to an account of interesting properties of the Lov\'{a}sz $\vartheta$-function in
\cite{Knuth94}, Chapter~11 of \cite{Lovasz19}, and more recently in Section~2.5 of \cite{Sason24}.

The following result generalizes Corollary~\ref{corollary: complete graphs} by relying on properties of the
Lov\'{a}sz $\vartheta$-function. For the third item of the next result, strongly regular graphs are first presented.

\begin{definition}[Strongly Regular Graphs]
\label{definition: strongly regular graphs}
{\em A regular graph $\Gr{G}$ that is neither complete nor empty is called a
{\em strongly regular} graph with parameters $(n,d,\lambda,\mu)$,
where $\lambda$ and $\mu$ are nonnegative integers, if the following conditions hold:
\begin{enumerate}[(1)]
\item \label{Item 1 - definition of SRG}
$\Gr{G}$ is a $d$-regular graph on $n$ vertices.
\item \label{Item 2 - definition of SRG}
Every two adjacent vertices in $\Gr{G}$ have exactly $\lambda$ common neighbors.
\item \label{Item 3 - definition of SRG}
Every two distinct and nonadjacent vertices in $\Gr{G}$ have exactly $\mu$ common neighbors.
\end{enumerate}
The family of strongly regular graphs with these four specified parameters is denoted by $\srg{n}{d}{\lambda}{\mu}$.
It is important to note that a family of the form $\srg{n}{d}{\lambda}{\mu}$ may contain multiple nonisomorphic strongly
regular graphs \cite{BrouwerM22}. We refer to a strongly regular graph as $\srg{n}{d}{\lambda}{\mu}$ if it belongs to this family.}
\end{definition}

\begin{proposition}
\label{proposition: Lovasz theta function}
{\em Let $\set{G}$ be an $\Gr{H}$-intersecting family of graphs on $n$ vertices, where $\Gr{H}$ is a nonempty,
simple, and undirected graph on $p$ vertices. The following holds:
\begin{enumerate}[(1)]
\item
\label{item 1: Lovasz}
\begin{align}
\label{eq1:19.12.24}
\log_2 \card{\set{G}} \leq \binom{n}{2} - \Biggl\lceil \max_{\bf{T}} \, \frac{\lambda_{\max}({\bf{T}})}{|\lambda_{\min}({\bf{T}})|} \Biggr\rceil,
\end{align}
where the maximization on the right-hand side of \eqref{eq1:19.12.24} is taken over all nonzero, symmetric $p \times p$ matrices
${\bf{T}} = (T_{i,j})$ with $T_{i,j}=0$ for all $i,j \in \OneTo{p}$ such that $\{i,j\} \notin \E{\Gr{H}}$ or $i=j$ (e.g., the adjacency
matrix of $\Gr{H}$), and $\lambda_{\max}({\bf{T}})$ and $\lambda_{\min}({\bf{T}})$ denote the largest and smallest (real) eigenvalues
of ${\bf{T}}$, respectively.
\item \label{item 2: Lovasz}
Specifically, if $\Gr{H}$ is a $d$-regular graph on $p$ vertices, where $d \in \OneTo{p-1}$, then
\begin{align}
\label{eq2:19.12.24}
\log_2 \card{\set{G}} \leq \binom{n}{2} - \Biggl\lceil \frac{d}{|\lambda_{\min}(\Gr{H})|} \Biggr\rceil,
\end{align}
where $\lambda_{\min}(\Gr{H})$ is the smallest eigenvalue of the adjacency matrix of the graph $\Gr{H}$.
\item \label{item 3: Lovasz}
If $\Gr{H}$ is a connected strongly regular graph in the family $\srg{p}{d}{\lambda}{\mu}$, then
\begin{align}
\label{eq3:19.12.24}
\log_2 \card{\set{G}} \leq \binom{n}{2} - \Biggl\lceil \frac{2d}{\sqrt{(\lambda-\mu)^2+4(d-\mu)}-\lambda+\mu} \Biggr\rceil.
\end{align}
\end{enumerate}}
\end{proposition}
\begin{proof}
By Corollary~\ref{corollary: Lovasz function},
\begin{align}
\log_2 \card{\set{G}} \leq \binom{n}{2} - \bigl( \lceil \vartheta(\CGr{H}) \rceil - 1 \bigr).
\end{align}
Item~\ref{item 1: Lovasz} then holds by the property that for every finite, simple, and undirected graph $\Gr{G}$
on $n$ vertices,
\begin{align}
\label{eq3:optimization}
\vartheta(\Gr{G}) = 1 + \max_{{\bf{T}}} \frac{\lambda_{\max}({\bf{T}})}{\bigl|\lambda_{\min}({\bf{T}})\bigr|},
\end{align}
where the maximization on the right-hand side of \eqref{eq3:optimization} is taken over all symmetric
nonzero $n \times n$ matrices ${\bf{T}} = (T_{i,j})$ with $T_{i,j}=0$ for all $i,j \in \OneTo{n}$ such that
$\{i,j\} \in \E{\Gr{G}}$ or $i=j$. Equality~\eqref{eq3:optimization} is then applied to the graph $\CGr{H}$,
so the maximization on the right-hand side of \eqref{eq1:19.12.24} is taken over all symmetric nonzero
$p \times p$ matrices ${\bf{T}} = (T_{i,j})$ with $T_{i,j}=0$ for all $i,j \in \OneTo{p}$ such that
$\{i,j\} \notin \E{\Gr{H}}$ or $i=j$. This includes in particular the adjacency matrix of the graph $\Gr{H}$, i.e.,
${\bf{T}} = \A(\Gr{H})$.

We next prove Item~\ref{item 2: Lovasz}, which refers to nonempty $d$-regular graphs on $p$ vertices. Relaxing
the bound in \eqref{eq1:19.12.24} by selecting ${\bf{T}}= \A(\Gr{H})$ gives $\lambda_{\max}({\bf{T}}) = d$, and
$\lambda_{\min}({\bf{T}}) = \lambda_{\min}(\Gr{H})$, which then gives the relaxed bound in \eqref{eq2:19.12.24}.

Item~\ref{item 3: Lovasz} follows from Item~\ref{item 2: Lovasz} by relying on the closed-form expression
of the smallest eigenvalue of the adjacency matrix of a connected strongly $d$-regular graph $\Gr{H}$ on $p$ vertices,
where each pair of adjacent vertices has exactly $\lambda$ common neighbors and each pair of nonadjacent vertices has
exactly $\mu$ common neighbors. Recall that a strongly regular graph is connected if and only if $\mu>0$.
In that case, the largest eigenvalue of the adjacency matrix is $\lambda_1(\Gr{H}) = d$ with multiplicity~1,
and the other two distinct eigenvalues of its adjacency matrix are given by (see, e.g., Chapter~21 of \cite{vanLintW02})
\begin{align}
\label{eigs-SRG}
r_{1,2} = \tfrac12 \, \Bigl( \lambda - \mu \pm \sqrt{ (\lambda-\mu)^2 + 4(d-\mu) } \, \Bigr),
\end{align}
with the respective multiplicities
\begin{align}
\label{eig-multiplicities-SRG}
m_{1,2} = \tfrac12 \, \Biggl( t-1 \mp \frac{2d+(t-1)(\lambda-\mu)}{\sqrt{(\lambda-\mu)^2+4(d-\mu)}} \, \Biggr).
\end{align}
Specifically, by \eqref{eigs-SRG}, the absolute value of the smallest eigenvalue of $\Gr{H}$ is given by
\begin{align}
\label{eq: smallest eig SRG}
\bigl|\lambda_{\min}(\Gr{H})\bigr| = \tfrac12 \, \Bigl( \sqrt{ (\lambda-\mu)^2 + 4(d-\mu) } + \mu -\lambda \, \Bigr).
\end{align}
Finally, substituting \eqref{eq: smallest eig SRG} into \eqref{eq2:19.12.24} gives \eqref{eq3:19.12.24}.
\end{proof}

\begin{remark}
{\em The derivation of \eqref{eq3:19.12.24} relies on \eqref{eq2:19.12.24}, where the latter is based on the lower bound on
$\vartheta(\CGr{H})$, obtained by relaxing the bound in \eqref{eq1:19.12.24} and selecting ${\bf{T}}= \A(\Gr{H})$ (see the
proof of Item~\ref{item 2: Lovasz} in Proposition~\ref{proposition: Lovasz theta function}).
Fortunately, the Lov\'{a}sz $\vartheta$-function of strongly regular graphs (and their complements, which are also strongly regular) is
known exactly, so there is no need for the lower bound on $\vartheta(\CGr{H})$ in this case. It is therefore of interest to examine
whether the bound in \eqref{eq3:19.12.24} can be improved by using \eqref{eq:23.01.2025} in combination with the exact value of
$\vartheta(\CGr{H})$ for a strongly regular graph $\Gr{H}$ in the family $\srg{p}{d}{\lambda}{\mu}$. In that case,
by Proposition~1 of \cite{Sason23}, we have
\begin{align}
\vartheta(\CGr{H}) &= 1 - \frac{d}{\lambda_{\min}(\Gr{H})}, \label{eq1:23.01.2025}
\end{align}
which also gives
\begin{align}
\vartheta(\CGr{H})
&= 1 + \frac{2d}{\sqrt{(\lambda-\mu)^2+4(d-\mu)}-\lambda+\mu}, \label{eq2:23.01.2025}
\end{align}
where \eqref{eq2:23.01.2025} holds by the expression of the smallest eigenvalue of the adjacency matrix of $\Gr{H}$, as given
by $r_2$ in \eqref{eigs-SRG}. Finally, combining inequality \eqref{eq:23.01.2025} with equality \eqref{eq2:23.01.2025} gives
exactly the same bound as in \eqref{eq3:19.12.24}. Thus, there is no improvement, and the bound remains identical in both approaches.
As a side note, interested readers are referred to a recent application of \eqref{eq2:23.01.2025}, which provides an alternative
proof of the friendship theorem in graph theory \cite{Sason25_friendship}.}
\end{remark}

It is natural to ask the following question:
\begin{question}
\label{question: matching}
{\em Is there a graph $\Gr{H}$, apart of $\CoG{2}$, for which the bound provided in Proposition~\ref{proposition: I.S., 2024}
is tight for a largest $\Gr{H}$-intersecting family of graphs?}
\end{question}

We provide a partial reply to Question~\ref{question: matching} by comparing the leftmost and rightmost terms
in \eqref{eq: UB and LB}, which is equivalent to comparing $\card{\hspace*{-0.05cm} \E{\Gr{H}}}$ and $\chrnum{\Gr{H}}-1$.
According to the inequality
\begin{align}
\label{eq: greedy coloring inequality}
\chrnum{\Gr{H}} \leq \Delta(\Gr{H})+1,
\end{align}
where $\Delta(\Gr{H})$ is the maximum degree of the vertices in the graph $\Gr{H}$, it
follows that unless $\Gr{H}$ is an edge, there exists a gap between the size of the graph
($\card{\hspace*{-0.05cm} \E{\Gr{H}}}$) and the chromatic number minus one ($\chrnum{\Gr{H}}-1$). Furthermore,
according to Brooks' theorem, for connected, undirected graphs
$\Gr{H}$ that are neither complete nor odd cycles, the chromatic number satisfies
\begin{align}
\label{eq: Brooks' Theorem}
\chrnum{\Gr{H}} \leq \Delta(\Gr{H}),
\end{align}
which provides a tighter bound in comparison to \eqref{eq: greedy coloring inequality},
further increasing the gap between $\card{\hspace*{-0.05cm} \E{\Gr{H}}}$ and $\chrnum{\Gr{H}}-1$ unless
$\Gr{H}$ is an edge. It is also noted that the chromatic number satisfies
\begin{align}
\chrnum{\Gr{H}} \leq \tfrac12 \bigl(1 + \sqrt{1 + 8 \, \card{\hspace*{-0.05cm} \E{\Gr{H}}}}\bigr),
\end{align}
with equality if and only if $\Gr{H}$ is a complete graph. This bound follows
from the observation that a graph with chromatic number $k$ must contain at least as many edges
as the complete graph on $k$ vertices. The chromatic number $\chrnum{\Gr{H}}$ cannot
therefore exceed the largest integer $k$ satisfying $\binom{k}{2} \leq \card{\hspace*{-0.05cm} \E{\Gr{H}}}$.
Consequently, if $\card{\hspace*{-0.05cm} \E{\Gr{H}}} \triangleq m \geq 2$, then the minimum possible gap between
$\chrnum{\Gr{H}}-1$ and $\card{\hspace*{-0.05cm} \E{\Gr{H}}}$ satisfies
\begin{align}
\card{\hspace*{-0.05cm} \E{\Gr{H}}} - \bigl(\chrnum{\Gr{H}} - 1 \bigr) &\geq \Bigl\lceil m - \tfrac12 \bigl(\sqrt{8m+1}-1 \bigr) \Bigr\rceil \geq 1,
\end{align}
which tends to infinity as $m \to \infty$.

\section{Number of graph homomorphisms}
\label{section: entropy bounds - Number of Graph Homomorphisms}

This section, composed of two independent parts, applies properties of Shannon entropy to derive bounds related to
the enumeration of graph homomorphisms. It offers additional insight into the interplay between combinatorial structures
and information-theoretic principles.

\subsection{An application of the probabilistic version of Shearer's lemma}
\label{subsection: application of the probabilistic version of Shearer's lemma}

The following known result relates the number of cliques of any two distinct orders in a graph.
\begin{proposition}
\label{proposition: 08.09.2024}
{\em Let $\Gr{G}$ be a finite, simple, and undirected graph on $n$ vertices, and let $m_\ell$ be the number of cliques
of order $\ell \in \naturals$ in $\Gr{G}$. Then, for all $s, t \in \naturals$ with $2 \leq s < t \leq n$,
\begin{align}
\label{eq5: 08.09.2024}
(t! \, m_t)^s \leq (s! \, m_s)^t.
\end{align}}
\end{proposition}

We next suggest a generalization of Proposition~\ref{proposition: 08.09.2024}.
\begin{proposition}
\label{proposition: 20.12.2024}
{\em Let $\Gr{G}$ be a finite, simple, and undirected graph on $n$ vertices,
let $s, t \in \naturals$ with $s < t < n$, let $\Gr{T}$ be an induced subgraph of
$\Gr{G}$ on $t$ vertices, and let $m(\Gr{H},\Gr{G})$ denote the number of
copies of a subgraph $\Gr{H}$ in the graph $\Gr{G}$. Then,
\begin{align}
\label{eq1a: 20.12.2024}
\bigl( t! \, m(\Gr{T}, \Gr{G}) \bigr)^s \leq \max_{\Gr{S}} \Bigl( s! \, m(\Gr{S}, \Gr{G}) \Bigr)^t,
\end{align}
where the maximization in \eqref{eq1a: 20.12.2024} is taken over all induced subgraphs $\Gr{S}$ of $\Gr{T}$ with $s$ vertices.

Let $\autcount{\Gr{H}}$ denote the size of the automorphism group of a graph $\Gr{H}$, defined as the number
of vertex permutations that preserve the graph's structure, i.e., its adjacency and non-adjacency relations.
Furthermore, let $\injcount{\Gr{H}}{\Gr{G}}$ denote the number of injective homomorphisms from $\Gr{H}$
to $\Gr{G}$ (i.e., homomorphisms where distinct vertices of $\Gr{H}$ map to distinct vertices of $\Gr{G}$).
Then, equivalently to \eqref{eq1a: 20.12.2024}, in terms of injective homomorphism counts,
\begin{align}
\label{eq1b: 20.12.2024}
\biggl(\frac{t!}{\autcount{\Gr{T}}} \cdot \injcount{\Gr{T}}{\Gr{G}} \biggr)^s \leq \max_{\Gr{S}} \,
\biggl(\frac{s!}{\autcount{\Gr{S}}} \cdot \injcount{\Gr{S}}{\Gr{G}} \biggr)^t,
\end{align}
where the maximization in \eqref{eq1b: 20.12.2024} is the same as in \eqref{eq1a: 20.12.2024}.}
\end{proposition}

\begin{proof}
\noindent
\begin{itemize}
\item Let $\V{\Gr{G}} = \OneTo{n}$, and let $\Gr{T}$
be an induced subgraph of $\Gr{G}$ with $\card{\hspace*{-0.05cm} \V{\Gr{T}}}=t < n$.
\item Select a copy of $\Gr{T}$ in $\Gr{G}$ uniformly at random, and then choose a
uniform random ordering of the vertices in that copy. This process produces a random
vector $(X_1, \ldots, X_t)$, representing the selected order of the vertices.
\item Let $m(\Gr{T},\Gr{G})$ denote the number of copies of $\Gr{T}$ in $\Gr{G}$. Then,
\begin{align}
\label{eq2: 25.11.2024}
\Ent{X_1, \ldots, X_t} = \log \bigl(t! \; m(\Gr{T},\Gr{G})\bigr),
\end{align}
as the vertices of each copy of $\Gr{T}$ in $\Gr{G}$ can be ordered in $t!$ equally
probable ways.
\item Let $\set{S}$ be chosen uniformly at random from all
subsets of $\OneTo{t}$ of fixed size $s$, where $1 \leq s<t$.~Then,
\begin{align}
\label{eq5: 25.11.2024}
\Prv{i \in \set{S}} = \frac{s}{t}, \quad \forall \, i \in \OneTo{t}.
\end{align}
\item By Proposition~\ref{proposition: Shearer's Lemma: 2nd version} and equalities
\eqref{eq2: 25.11.2024} and \eqref{eq5: 25.11.2024}, it follows that
\begin{align}
\label{eq6: 25.11.2024}
\bigExpecwrt{\set{S}}{\Ent{X_{\set{S}}}} \geq \frac{s}{t} \cdot \log \bigl(t! \; m(\Gr{T},\Gr{G})\bigr).
\end{align}
\item The random subvector $X_{\set{S}}$ corresponds to a copy, in $\Gr{G}$, of an induced subgraph $\Gr{S} \subseteq \Gr{T}$
with $s$ vertices.
All $s!$ permutations of the subvector $X_{\set{S}}$ correspond to the same copy of $\Gr{S}$ in $\Gr{G}$, and there
are $m(\Gr{S},\Gr{G})$ such copies of $\Gr{S}$ in $\Gr{G}$.
\item The entropy of the random subvector $X_{\set{S}}$ therefore satisfies
\begin{align}
\label{eq7: 25.11.2024}
\Ent{X_{\set{S}}} \leq \log \bigl(s! \; m(\Gr{S},\Gr{G})\bigr),
\end{align}
where $m(\Gr{S},\Gr{G})$ denotes the number of copies of a graph $\Gr{S}$ in $\Gr{G}$, and $s!$ accounts for the $s!$ permutations
of the vector $X_{\set{S}}$ that correspond to the same copy of $\Gr{S}$ in $\Gr{G}$.
\item By \eqref{eq7: 25.11.2024}, it follows that
\begin{align}
\label{eq8: 25.11.2024}
\bigExpecwrt{\set{S}}{\Ent{X_{\set{S}}}} \leq \max_{\Gr{S}} \log \bigl(s! \; m(\Gr{S},\Gr{G})\bigr),
\end{align}
where the maximization on the right-hand side of \eqref{eq8: 25.11.2024} is taken over all
induced subgraphs $\Gr{S}$ of $\Gr{T}$ on $s$ vertices.
\item Combining \eqref{eq6: 25.11.2024} and \eqref{eq8: 25.11.2024} yields
\begin{align}
\label{eq9: 25.11.2024}
\frac{s}{t} \cdot \log \bigl(t! \; m(\Gr{T},\Gr{G})\bigr) \leq \max_{\Gr{S}} \, \log \bigl(s! \; m(\Gr{S},\Gr{G})\bigr),
\end{align}
and exponentiating both sides of \eqref{eq9: 25.11.2024} gives \eqref{eq1a: 20.12.2024}.
\item The following equality holds:
\begin{align}
\label{eq3: 25.11.2024}
\injcount{\Gr{H}}{\Gr{G}} = \autcount{\Gr{H}} \, m(\Gr{H},\Gr{G}),
\end{align}
since each copy of $\Gr{H}$ in $\Gr{G}$ corresponds to exactly $\autcount{\Gr{H}}$ distinct injective homomorphisms from $\Gr{H}$
to $\Gr{G}$, as the vertices of $\Gr{H}$ can be labeled in $\autcount{\Gr{H}}$ different ways that preserve graph isomorphism.
Combining \eqref{eq1a: 20.12.2024} and \eqref{eq3: 25.11.2024} yields inequality \eqref{eq1b: 20.12.2024}.
Furthermore, by \eqref{eq3: 25.11.2024}, inequalities \eqref{eq1a: 20.12.2024} and \eqref{eq1b: 20.12.2024} are equivalent.
\end{itemize}
\end{proof}

\begin{remark}[Specialization of Proposition~\ref{proposition: 20.12.2024}]
{\em Proposition~\ref{proposition: 20.12.2024} can be specialized to Proposition~\ref{proposition: 08.09.2024}
by setting $\Gr{T} = \CoG{t}$ (a clique of order $t$), for which every induced subgraph of $\Gr{T}$ on $s$
vertices is a clique $\Gr{S}$ of order $s$ ($\Gr{S} = \CoG{s}$). In that case, $\autcount{\Gr{T}} = t!$ and $\autcount{\Gr{S}} = s!$.
Consequently, the maximization on the right-hand side of \eqref{eq1b: 20.12.2024} is performed over the single graph $\CoG{s}$,
which gives
\begin{align}
\label{eq13: 25.11.2024}
\injcount{\CoG{t}}{\Gr{G}}^s \leq \injcount{\CoG{s}}{\Gr{G}}^t, \quad 1 \leq s<t<n.
\end{align}
By \eqref{eq3: 25.11.2024}, we have
\begin{align}
\label{eq14a: 25.11.2024}
\injcount{\CoG{t}}{\Gr{G}} = t! \, m_t,  \\
\label{eq14b: 25.11.2024}
\injcount{\CoG{s}}{\Gr{G}} = s! \, m_s,
\end{align}
where $m_t$ and $m_s$ denote, respectively, the number of cliques of orders $t$ and $s$ in $\Gr{G}$.
Combining \eqref{eq13: 25.11.2024}, \eqref{eq14a: 25.11.2024}, and \eqref{eq14b: 25.11.2024} then gives
\begin{align}
\label{eq15: 25.11.2024}
(t! \, m_t)^s \leq (s! \, m_s)^t,   \quad 1 \leq s<t<n.
\end{align}
This reproduces Proposition~\ref{proposition: 08.09.2024}, which establishes a relationship between
the numbers of cliques of two different orders in a finite, simple, undirected graph $\Gr{G}$.}
\end{remark}

\subsection{An entropy-based proof for bounding the number of graph homomorphisms}
\label{subsection: an entropy-based proof for bounding the number of graph homomorphisms}

Perfect graphs are characterized by the property that, for each of their induced subgraphs, the
chromatic number and clique number coincide. The complement of a
perfect graph is perfect as well. Perfect graphs include many important families of graphs such
as bipartite graphs, line graphs of bipartite graphs, chordal graphs, comparability graphs, and
the complements of all these graphs \cite{Lovasz83,Trotignon15}.

A complete bipartite graph, denoted by $\CoBG{s}{t}$ for $s,t \in \naturals$, consists of two
partite sets of sizes $s$ and $t$, where every vertex in one partite set is adjacent to all the
vertices in the other partite set. It is in particular a perfect graph.

In the following, we rely on Proposition~\ref{proposition: UB on the number of homomorphisms}
to derive an upper bound on the number of homomorphisms from a perfect graph to a graph.
We then rely on properties of Shannon entropy to derive a lower bound on the number of homomorphisms
from any complete bipartite graph to any bipartite graph, and examine its tightness by
comparing it to the specialized upper bound that is based on Proposition~\ref{proposition: UB on the number of homomorphisms}.
\begin{proposition}[Number of graph homomorphisms]
\label{proposition: bounds on the number of homomorphisms}
{\em Let $\Gr{T}$ and $\Gr{G}$ be simple, finite, and undirected graphs with no isolated vertices, and
suppose that $\Gr{T}$ is also perfect. Then, the number of homomorphisms from $\Gr{T}$ to $\Gr{G}$
satisfies
\begin{align}
\label{eq4: 16.09.2024}
\homcount{\Gr{T}}{\Gr{G}} \leq \bigl( 2 \, \card{\hspace*{-0.05cm} \E{\Gr{G}}} \bigr)^{\vartheta(\Gr{T})},
\end{align}
where $\vartheta(\Gr{T})$ denotes the Lov\'{a}sz $\vartheta$-function of the graph $\Gr{T}$ \cite{Lovasz79_IT}.

Furthermore, let $\Gr{G}$ be a simple bipartite graph with no isolated vertices, whose partite sets have sizes
$n_1$ and $n_2$, and suppose that the number of edges in $\Gr{G}$ is equal to $\alpha n_1 n_2$ for some $\alpha \in (0,1]$.
Then, for all positive integers $s, t \in \naturals$, the number of homomorphisms from the complete bipartite
graph $\CoBG{s}{t}$ to $\Gr{G}$ satisfies
\begin{align}
\label{eq4b: 16.09.2024}
\alpha^{st} \, \min\{n_1, n_2\}^{-|s-t|} \, (n_1 n_2)^{\max\{s,t\}} \leq \homcount{\CoBG{s}{t}}{\Gr{G}}
\leq \bigl( 2 \alpha n_1 n_2 \bigr)^{\max\{s,t\}}.
\end{align}}
\end{proposition}

\begin{proof}
For a perfect graph $\Gr{T}$, we have $\indnum{\Gr{T}} = \vartheta(\Gr{T}) = \findnum{\Gr{T}} = \chrnum{\CGr{T}}$,
which by \eqref{eq1: 11.09.2024} gives \eqref{eq4: 16.09.2024}.

We next prove the rightmost inequality in \eqref{eq4b: 16.09.2024}.
Every bipartite graph is a perfect graph, so it follows that
$\indnum{\CoBG{s}{t}} = \vartheta(\CoBG{s}{t}) = \findnum{\CoBG{s}{t}} = \chrnum{\overline{\CoBG{s}{t}}}$.
The independence number of the complete bipartite graph $\CoBG{s}{t}$ is the size of the largest among
its two partite vertex sets, so $\indnum{\CoBG{s}{t}} = \max\{s,t\}$.
Similarly, the complement graph $\overline{\CoBG{s}{t}} = \CoG{s} \cup \CoG{t}$ is the disjoint union of the
complete graphs $\CoG{s}$ and $\CoG{t}$, so its chromatic number is given by
$\chrnum{\overline{\CoBG{s}{t}}} = \max\{s,t\}$. Consequently, $\vartheta(\CoBG{s}{t}) = \max\{s,t\}$, whose
substitution into the right-hand side of \eqref{eq4: 16.09.2024}, along with
$\card{\hspace*{-0.05cm} \E{\Gr{G}}} = \alpha n_1 n_2$ where $\alpha \in (0,1]$ (by assumption), gives the
rightmost inequality in~\eqref{eq4b: 16.09.2024}.

We finally prove the leftmost inequality in \eqref{eq4b: 16.09.2024}.
Let $\set{U}$ and $\set{V}$ denote the partite vertex sets of the simple bipartite graph $\Gr{G}$, where
$\card{\set{U}} = n_1$ and $\card{\set{V}} = n_2$.
Let $(U,V)$ be a random vector taking values in $\set{U} \times \set{V}$, and suppose that $\{U,V\}$ is
distributed uniformly at random on the edges of $\Gr{G}$. Then, the joint entropy of $(U,V)$ is given by
\begin{align}
\Ent{U,V} &= \log \, \bigl| \E{\Gr{G}} \bigr| \nonumber \\
&=\log(\alpha n_1 n_2).  \label{eq: Ent U,V}
\end{align}
The random vector $(U,V)$ can be sampled by first sampling $U=u$ from the marginal probability mass function (PMF)
of $U$, denoted by $\pmfOf{U}$, and then sampling $V$ from the conditional PMF $\CondpmfOf{V}{U}(\cdot|u)$.
We next construct a random vector $(U_1, \ldots, U_s, V_1, \ldots, V_t)$ as follows:
\begin{itemize}
\item Let $V_1, \ldots, V_t$ be conditionally independent and identically distributed (i.i.d.) given $U$, having the
conditional PMF
\begin{align}
\CondpmfOf{V_1, \ldots, V_t}{U}(v_1, \ldots, v_t|u) = \prod_{j=1}^t \CondpmfOf{V}{U}(v_j|u), \quad \forall \, u \in \set{U},
\; \; (v_1, \ldots, v_t) \in \set{V}^{\, t}.   \label{eq1: cond. PMF}
\end{align}
\item Let $U_1, \ldots, U_s$ be conditionally i.i.d. given $(V_1, \ldots, V_t)$, having the conditional PMF
\begin{align}
& \hspace*{-0.3cm} \CondpmfOf{U_1, \ldots, U_s}{V_1, \ldots, V_t}(u_1, \ldots, u_s|v_1, \ldots, v_t) \nonumber \\
\label{eq2a: cond. PMF}
&= \prod_{i=1}^s
\CondpmfOf{U_i}{V_1, \ldots, V_t}(u_i|v_1, \ldots, v_t), \quad \forall \, (u_1, \ldots, u_s) \in \set{U}^{\, s}, \; \;
(v_1, \ldots, v_t) \in \set{V}^{\, t},
\end{align}
where the conditional PMFs on the right-hand side of \eqref{eq2a: cond. PMF} are given by
\begin{align}
& \hspace*{-0.3cm} \CondpmfOf{U_i}{V_1, \ldots, V_t}(u|v_1, \ldots, v_t) \nonumber \\
&= \frac{\pmfOf{U}(u) \,
\overset{t}{\underset{j=1}{\prod}} \CondpmfOf{V}{U}(v_j|u)}{\underset{u' \in \set{U}}{\sum} \biggl\{
\pmfOf{U}(u') \, \overset{t}{\underset{j=1}{\prod}} \CondpmfOf{V}{U}(v_j|u') \biggr\}},
\quad \forall \, u \in \set{U},
\; \; (v_1, \ldots, v_t) \in \set{V}^{\, t}, \; \; i \in \OneTo{s}.  \label{eq2b: cond. PMF}
\end{align}
\end{itemize}
By the construction of the random vector $(U_1, \ldots, U_s, V_1, \ldots, V_t)$ in \eqref{eq1: cond. PMF}--\eqref{eq2b: cond. PMF}, the following holds:
\begin{enumerate}
\item The random variables $U_1, \ldots, U_s$ are identically distributed, and $U_i \sim U$ (i.e., $\pmfOf{U_i} = \pmfOf{U}$) for all $i \in \OneTo{s}$.
Indeed, it first follows from \eqref{eq1: cond. PMF} that
\begin{align}
\pmfOf{V_1, \ldots, V_t}(v_1, \ldots, v_t) = \sum_{u \in \set{U}} \biggl\{ \pmfOf{U}(u) \, \overset{t}{\underset{j=1}{\prod}} \CondpmfOf{V}{U}(v_j|u)
\biggr\} \quad \forall \, (v_1, \ldots, v_t) \in \set{V}^{\, t}.  \label{eq2c: cond. PMF}
\end{align}
Hence, for all $i \in \OneTo{s}$ and $u \in \set{U}$,
\begin{align}
\pmfOf{U_i}(u) &= \sum_{(v_1, \ldots, v_t) \in \set{V}^{\, t}} \biggl\{ \CondpmfOf{U_i}{V_1, \ldots, V_t}(u|v_1, \ldots, v_t)
\, \pmfOf{V_1, \ldots, V_t}(v_1, \ldots, v_t) \biggr\} \label{eq1: 17.3.25} \\
&= \sum_{(v_1, \ldots, v_t) \in \set{V}^{\, t}} \biggl\{ \pmfOf{U}(u) \,
\overset{t}{\underset{j=1}{\prod}} \CondpmfOf{V}{U}(v_j|u) \biggr\} \label{eq2: 17.3.25} \\
&= \pmfOf{U}(u) \, \overset{t}{\underset{j=1}{\prod}} \biggl\{ \sum_{v_j \in \set{V}} \CondpmfOf{V}{U}(v_j|u) \biggr\} \label{eq3: 17.3.25} \\
&= \pmfOf{U}(u), \label{eq4: 17.3.25}
\end{align}
where \eqref{eq1: 17.3.25} holds by Bayes' rule; \eqref{eq2: 17.3.25} holds by combining \eqref{eq2b: cond. PMF} and \eqref{eq2c: cond. PMF};
\eqref{eq3: 17.3.25} holds by expressing the outer $t$-dimensional summation over $\set{V}^{\, t}$ as the product of $t$ inner one-dimensional
summations over $\set{V}$, due to the product form on the right-hand side of \eqref{eq2: 17.3.25}, and finally \eqref{eq4: 17.3.25} holds since
the conditional probability masses in each inner summation on the right-hand side of \eqref{eq3: 17.3.25} add to~1.
\item For all $i \in \OneTo{s}$ and $j \in \OneTo{t}$, we have $(U_i, V_j) \sim (U,V)$, and further $(U_i, V_1, \ldots, V_t) \sim (U, V_1, \ldots, V_t)$.
Indeed, combining \eqref{eq1: cond. PMF}, \eqref{eq2b: cond. PMF}, and \eqref{eq2c: cond. PMF} yields (by another application of Bayes' rule)
\begin{align}
\pmfOf{U_i, V_1, \ldots, V_t}(u, v_1, \ldots, v_t) &= \pmfOf{U}(u) \,
\overset{t}{\underset{j=1}{\prod}} \CondpmfOf{V}{U}(v_j|u) \nonumber \\
\label{eq5: 17.3.25}
&= \pmfOf{U, V_1, \ldots, V_t}(u, v_1, \ldots, v_t), \quad \forall \, u \in \set{U}, \; \; (v_1, \ldots, v_t) \in \set{V}^{\, t}, \; \; i \in \OneTo{s}.
\end{align}
Then, a marginalization of \eqref{eq5: 17.3.25} by summing over all $(v_1, \ldots, v_{j-1}, v_{j+1}, \ldots, v_t) \in \set{V}^{\, t-1}$ gives
\begin{align}
\label{eq6: 17.3.25}
\pmfOf{U_i,V_j}(u,v) = \pmfOf{U,V}(u,v), \quad \forall \, i \in \OneTo{s}, \; j \in \OneTo{t}, \;
u \in \set{U}, \; v \in \set{V}.
\end{align}
\end{enumerate}

The joint entropy of the random subvector $(U_1, V_1, \ldots, V_t)$ then satisfies
\begin{align}
\Ent{U_1, V_1, \ldots, V_t} &= \Ent{U_1} + \sum_{j=1}^t \EntCond{V_j}{U_1} \label{eq0: 16.03.2025} \\
&= \Ent{U} + t \EntCond{V}{U} \label{eq0a: 16.03.2025} \\
&= t \Ent{U,V} - (t-1) \Ent{U} \label{eq0b: 16.03.2025} \\
&= t \log(\alpha n_1 n_2) - (t-1) \Ent{U} \label{eq0c: 16.03.2025} \\
&\geq t \log(\alpha n_1 n_2) - (t-1) \log n_1 \label{eq1: 16.03.2025} \\
&= \log(\alpha^t n_1 n_2^t),  \label{eq5: 16.09.2024}
\end{align}
where \eqref{eq0: 16.03.2025} holds by the chain rule of Shannon entropy, since (by construction)
$V_1, \ldots, V_t$ are conditionally independent given $U$ (see \eqref{eq1: cond. PMF}) and also since
$(U_1, V_1, \ldots, V_t) \sim (U, V_1, \ldots, V_t)$ (see \eqref{eq5: 17.3.25} with $i=1$);
\eqref{eq0a: 16.03.2025} relies on \eqref{eq6: 17.3.25};
\eqref{eq0b: 16.03.2025} holds by a second application of the chain rule;
\eqref{eq0c: 16.03.2025} holds by \eqref{eq: Ent U,V}, and finally \eqref{eq1: 16.03.2025} follows from
the uniform bound, which states that if $X$ is a discrete random variable supported on a finite set
$\set{S}$, then $\Ent{X} \leq \log \card{\set{S}}$. In this case,
$\Ent{U} \leq \log \card{\set{U}} = \log n_1$.
Consequently, the joint entropy of the random vector $(U_1, \ldots, U_s, V_1, \ldots, V_t)$ satisfies
\begin{align}
\Ent{U_1, \ldots, U_s, V_1, \ldots, V_t}
&= \Ent{V_1, \ldots, V_t} + \sum_{i=1}^s \EntCond{U_i}{V_1, \ldots, V_t} \label{eq2a: 16.03.2025} \\
&= \Ent{V_1, \ldots, V_t} + s \EntCond{U_1}{V_1, \ldots, V_t} \label{eq2b: 16.03.2025} \\[0.1cm]
&= s \bigl[ \Ent{V_1, \ldots, V_t} + \EntCond{U_1}{V_1, \ldots, V_t} \bigr] - (s-1) \Ent{V_1, \ldots, V_t} \nonumber \\[0.1cm]
&= s \Ent{U_1, V_1, \ldots, V_t} - (s-1) \Ent{V_1, \ldots, V_t} \label{eq2c: 16.03.2025} \\
&\geq s \log(\alpha^t n_1 n_2^t) - (s-1) \Ent{V_1, \ldots, V_t} \label{eq:13.03.2025}  \\
&\geq s \log(\alpha^t n_1 n_2^t) - (s-1) \log(n_2^t) \label{eq3: 16.03.2025} \\
&= \log(\alpha^{st} n_1^s n_2^t),  \label{eq6: 16.09.2024}
\end{align}
where \eqref{eq2a: 16.03.2025} holds by the chain rule and since (by construction) the random variables
$U_1, \ldots, U_s$ are conditionally independent given $V_1, \ldots, V_t$ (see \eqref{eq2a: cond. PMF});
\eqref{eq2b: 16.03.2025} holds since, by construction, all the $U_i$'s ($i \in \OneTo{s}$)
are identically conditionally distributed given $(V_1, \ldots, V_t)$ (see \eqref{eq2b: cond. PMF});
\eqref{eq2c: 16.03.2025} holds by another use of the chain rule; \eqref{eq:13.03.2025} holds by
\eqref{eq5: 16.09.2024}, and finally \eqref{eq3: 16.03.2025} holds by the uniform bound which implies
in this case that $\Ent{V_1, \ldots, V_t} \leq \log(\card{\set{V}}^{\, t}) = \log(n_2^t)$.

Each vector $(U_1, \ldots, U_s, V_1, \ldots, V_t)$ can be mapped to a homomorphism from $\CoBG{s}{t}$ to $\Gr{G}$
via an injective mapping. To that end, label the vertices of the complete bipartite graph $\CoBG{s}{t}$ by
the elements of $\OneTo{s+t}$, assigning the labels $1, \ldots, s$ to the vertices in the partite set of size
$s$, and the labels $s+1, \ldots, s+t$ to the vertices in the partite set of size $t$.
Then, for all $i \in \OneTo{s}$ and $j \in \OneTo{t}$, map each edge $\{i, i+j\} \in \E{\CoBG{s}{t}}$
to $\{U_i, V_j\} \in \E{\Gr{G}}$. This defines a homomorphism $\CoBG{s}{t} \to \Gr{G}$ since
$\{U_i, V_j\} \in \E{\Gr{G}}$ holds by construction; see \eqref{eq2b: cond. PMF}. Recall that in \eqref{eq2b: cond. PMF},
$\{U,V\}$ is uniformly distributed over the edges of the graph $\Gr{G}$, where $U \in \set{U}$ and $V \in \set{V}$
(by construction), $\pmfOf{U}$ denotes the marginal PMF of $U$, and $\CondpmfOf{V}{U}$ denotes the conditional PMF
of $V$ given $U$. The suggested mapping is injective since it maps distinct such vectors to distinct homomorphisms in
$\Hom{\CoBG{s}{t}}{\Gr{G}}$. By \eqref{eq: number of homomorphisms} and the uniform bound, it then follows that
\begin{align}
\label{eq7: 16.09.2024}
\Ent{U_1, \ldots, U_s, V_1, \ldots, V_t} \leq \log \bigl(\homcount{\CoBG{s}{t}}{\Gr{G}}\bigr).
\end{align}
Combining \eqref{eq6: 16.09.2024} and \eqref{eq7: 16.09.2024} yields
\begin{align}
\label{eq8: 16.09.2024}
\homcount{\CoBG{s}{t}}{\Gr{G}} \geq \alpha^{st} n_1^s n_2^t.
\end{align}
The right-hand side of \eqref{eq8: 16.09.2024} is not necessarily symmetric in $n_1$ and $n_2$ (or in $s,t \in \naturals$). Consequently,
swapping either $n_1$ and $n_2$ (or $s$ and $t$) gives
\begin{align}
\label{eq9: 16.09.2024}
\homcount{\CoBG{s}{t}}{\Gr{G}} & \geq \max \Bigl\{ \alpha^{st} n_1^s n_2^t, \;  \alpha^{st} n_1^t n_2^s \Bigr\} \nonumber \\
&= \alpha^{st} \, \min\{n_1, n_2\}^{\min\{s,t\}} \, \max\{n_1, n_2\}^{\max\{s,t\}} \nonumber \\
&= \alpha^{st} \, \min\{n_1, n_2\}^{\min\{s,t\}-\max\{s,t\}} \, (n_1 n_2)^{\max\{s,t\}} \nonumber \\
&= \alpha^{st} \, \min\{n_1, n_2\}^{-|s-t|} \, (n_1 n_2)^{\max\{s,t\}},
\end{align}
which proves the leftmost inequality in \eqref{eq4b: 16.09.2024}.
\end{proof}

Setting $s=t$ in Proposition~\ref{proposition: bounds on the number of homomorphisms} gives
the following.
\begin{corollary}
\label{corollary: bounds on the number of homomorphisms}
{\em Let $\Gr{G}$ be a simple bipartite graph with partite sets of sizes $n_1$ and $n_2$, no isolated vertices,
and $\alpha n_1 n_2$ edges for some $\alpha \in (0,1]$. Then, for all $s \in \naturals$,
\begin{align}
\label{eq4c: 16.09.2024}
\alpha^{s^2} (n_1 n_2)^s \leq \homcount{\CoBG{s}{s}}{\Gr{G}} \leq (2 \alpha)^s (n_1 n_2)^s.
\end{align}
Consequently, for a fixed $\alpha \in (0,1]$, the number of homomorphisms from the complete bipartite graph
$\CoBG{s}{s}$ to $\Gr{G}$ scales like
$(n_1 n_2)^s$.}
\end{corollary}

\begin{remark}[Comparison to Sidorenko's lower bound]
\label{remark: Comparison to Sidorenko's LB}
{\em A graph $\Gr{H}$ is said to be Sidorenko if it has the property that for every graph $\Gr{G}$
\begin{align}
\label{def: Sidorenko graph}
\frac{\homcount{\Gr{H}}{\Gr{G}}}{\bigcard{\V{\Gr{G}}}^{\; \card{\V{\Gr{H}}}}} \geq
\biggl( \frac{2 \, \card{\E{\Gr{G}}}}{\card{\V{\Gr{G}}}^{\, 2}} \biggr)^{\, \card{\E{\Gr{H}}}}.
\end{align}
Sidorenko's conjecture states that every bipartite graph is Sidorenko (introduced by Sidorenko \cite{Sidorenko93}
and by Erd\v{o}s-Simonovits \cite{Simonovits84}). Although this conjecture remains an open problem
in its full generality, it is known that every bipartite graph containing a vertex adjacent
to all vertices in its other part is Sidorenko (see, e.g., \cite[Theorem~5.5.14]{Zhao23}, originally proved in
\cite{ConlonFS10}, and simplified in \cite{ConlonFS10b}). In particular, every complete bipartite graph is Sidorenko (see \cite[Theorem~5.5.12]{Zhao23}).
Specializing \eqref{def: Sidorenko graph} to a complete bipartite graph $\Gr{H} = \CoBG{s}{t}$, where $s,t \in \naturals$,
yields the following inequality for every graph $\Gr{G}$:
\begin{align}
\label{eq1: 21.05.2025}
\frac{\homcount{\CoBG{s}{t}}{\Gr{G}}}{\bigcard{\V{\Gr{G}}}^{s+t}} \geq
\biggl( \frac{2 \, \card{\E{\Gr{G}}}}{\card{\V{\Gr{G}}}^{\, 2}} \biggr)^{st}.
\end{align}
Let us now further specialize inequality \eqref{eq1: 21.05.2025} to the case where $\Gr{G}$ is a simple bipartite graph with
partite sets of sizes $n_1$ and $n_2$, has no isolated vertices, and contains $\alpha n_1 n_2$ edges for some $\alpha \in (0,1]$.
In this setting, straightforward algebra gives
\begin{align}
\label{eq2: 21.05.2025}
\homcount{\CoBG{s}{t}}{\Gr{G}} \geq (2 \alpha)^{st} (n_1+n_2)^{s+t-2st} (n_1 n_2)^{st} \triangleq \mathrm{LB}_1.
\end{align}
This lower bound on $\homcount{\CoBG{s}{t}}{\Gr{G}}$ is compared next to the bound
\begin{align}
\label{eq3: 21.05.2025}
\homcount{\CoBG{s}{t}}{\Gr{G}} \geq \alpha^{st} \, \min\{n_1, n_2\}^{-|s-t|} \, (n_1 n_2)^{\max\{s,t\}} \triangleq \mathrm{LB}_2,
\end{align}
which appears as the leftmost inequality in \eqref{eq4b: 16.09.2024}.
To compare these two bounds, we examine the ratio $\frac{\mathrm{LB}_2}{\mathrm{LB}_1}$.
Combining \eqref{eq2: 21.05.2025} and \eqref{eq3: 21.05.2025} gives
\begin{align}
\label{eq4: 21.05.2025}
\frac{\mathrm{LB}_2}{\mathrm{LB}_1} = 2^{-st} \, (n_1+n_2)^{2st-(s+t)} \; (n_1 n_2)^{\, \max\{s,t\}-st} \; \min\{n_1, n_2\}^{-|s-t|}.
\end{align}
The right-hand side of \eqref{eq4: 21.05.2025} is symmetric in $n_1$ and $n_2$, and also in $s$ and $t$.
Without loss of generality, suppose that $s \geq t$, and let $\delta \triangleq \frac{\max\{n_1,n_2\}}{\min\{n_1,n_2\}} \geq 1$.
By \eqref{eq4: 21.05.2025}, it can be verified to give
\begin{align}
\label{eq5: 21.05.2025}
\frac{\mathrm{LB}_2}{\mathrm{LB}_1} &= 2^{-s} \, \Biggl( \frac{(1+\delta)^2}{2 \delta} \Biggr)^{s(t-1)} \; (1+\delta)^{s-t} \\
&\geq 2^{-s} \, 2^{s(t-1)} (1+\delta)^{s-t} \nonumber \\
&\geq 2^{s(t-2)} \, 2^{s-t} \nonumber \\
\label{eq6: 21.05.2025}
&= 2^{st-(s+t)}.
\end{align}
Since the right-hand side of \eqref{eq6: 21.05.2025} is symmetric in $s$ and $t$, the earlier assumption that $s \geq t$ can be dropped. Consequently,
\begin{enumerate}[(1)]
\item If $s=1$ or $t=1$ (i.e., $\CoBG{s}{t}$ is a star graph), then it follows from \eqref{eq6: 21.05.2025} that $\mathrm{LB}_2 \geq \tfrac12 \, \mathrm{LB}_1$;
\item If $s=t=2$, then $\mathrm{LB}_2 \geq \mathrm{LB}_1$; if also $\delta>1$ (i.e., $n_1 \neq n_2$), then \eqref{eq5: 21.05.2025}
yields $\mathrm{LB}_2 > \mathrm{LB}_1$;
\item Otherwise (i.e., if $s,t \geq 2$ and $\max\{s,t\} \geq 3$), we have $st-(s+t) \geq 1$, and therefore $\mathrm{LB}_2 \geq 2^{st-(s+t)} \, \mathrm{LB}_1
\geq 2 \, \mathrm{LB}_1$.
\end{enumerate}
To conclude, except for the case where $\CoBG{s}{t}$ is a star graph, our lower bound on $\homcount{\CoBG{s}{t}}{\Gr{G}}$ in \eqref{eq3: 21.05.2025}
compares favourably to Sidorenko's lower bound in \eqref{eq2: 21.05.2025}.
}
\end{remark}

\section*{Acknowledgments}
The author acknowledges the timely and helpful reports of the two referees, and a
stimulating discussion with Yuval Peled during the author's seminar talk on the subject at
the Einstein Institute of Mathematics, Hebrew University of Jerusalem.
The author also appreciates the hospitality provided during the seminar, which was organized by Yuval.


\begin{thebibliography}{99}

\bibitem{SimonovitsS76}
M. Simonovits and V. T. S\'{o}s, Intersection theorems for graphs II. In {\em Combinatorics (Proc. Fifth Hungarian Colloq., Keszthely, 1976)},
vol.~II, pp. 1017--1030, Colloq. Math. Soc. J\'{a}nos Bolyai, 18. North-Holland Publishing Co., Amsterdam-New York.
\url{https://users.renyi.hu/~sos/1976_Intersection_Theorems_for_Graphs_II.pdf}

\bibitem{SimonovitsS78}
M. Simonovits and V. T. S\'{o}s, ``Intersection theorems for graphs,'' {\em Problemes combinatoires et théorie
des graphes (Colloques internationaux du Centre National de la Recherche Scientifique (CNRS)}, University of Orsay,
Orsay, France, July~1976), pp. 389--391, {\em Colloques internationaux du CNRS}, no.~260, Paris, France, 1978.
\url{https://users.renyi.hu/~miki/OrsayB.pdf}

\bibitem{SimonovitsS80}
M. Simonovits and V. T. S\'{o}s, ``Intersection theorems on structures,'' {\em Annals of Discrete Mathematics},
vol.~6, pp.~301--313, 1980.  \url{https://doi.org/10.1016/S0167-5060(08)70715-5}

\bibitem{ChungGFS86}
F. R. K. Chung, L. R. Graham, P. Frankl and J. B. Shearer, ``Some intersection
theorems for ordered sets and graphs," {\em Journal of Combinatorial Theory,
Series~A}, vol.~43, no.~1, pp.~23--37, 1986. \url{https://doi.org/10.1016/0097-3165(86)90019-1}

\bibitem{EllisFF12}
D. Ellis, Y. Filmus, and E. Friedgut, ``Triangle-intersecting families of graphs,''
{\em Journal of the European Mathematical Society}, vol.~14, no.~3, pp.~841--855, 2012.
\url{https://doi.org/10.4171/JEMS/320}

\bibitem{Ellis22}
D. Ellis, ``Intersection problems in extremal combinatorics: theorems, techniques and
questions - old and new,'' {\em Surveys in Combinatorics 2022}, pp.~115--173. London
Mathematical Society Lecture Note Series, Cambridge University Press, 2022.
\url{https://doi.org/10.1017/9781009093927.005}

\bibitem{BergerZ23}
A. Berger and Y. Zhao, ``$\mathrm{K}_4$-intersecting families of graphs,'' {\em Journal of
Combinatorial Theory, Series B}, vol.~163, pp.~112--132, November 2023.
\url{https://doi.org/10.1016/j.jctb.2023.07.005}

\bibitem{KellerL19}
N. Keller and N. Lifshitz, ``A note on large $\mathrm{H}$-intersecting families,'' {\em SIAM Journal on
Discrete Mathematics}, vol.~33, no.~1, pp.~398--401, 2019. \url{https://doi.org/10.1137/18M1220765}

\bibitem{AignerZ18}
M. Aigner and G. M. Ziegler, {\em Proofs from THE BOOK}, Sixth Edition,
Springer, Berlin, 2018. \url{https://doi.org/10.1007/978-3-662-57265-8}

\bibitem{Jukna11}
S. Jukna, {\em Extremal Combinatorics with Applications in Computer Science}, second edition,
Springer, 2011. \url{https://doi.org/10.1007/978-3-642-17364-6}

\bibitem{Galvin14}
D. Galvin, ``Three tutorial lectures on entropy and counting,'' {\em Proceedings of the 1st Lake
Michigan Workshop on Combinatorics and Graph Theory}, Kalamazoo, MI, USA, March 2014.
\url{https://doi.org/10.48550/arXiv.1406.7872}

\bibitem{Pippenger77}
N. Pippenger, ``An information-theoretic method in combinatorial theory,'' {\em Journal of Combinatorial
Theory, Series~A}, vol.~23, no.~1, pp.~99--104, July 1977.
\url{https://doi.org/10.1016/0097-3165(77)90083-8}

\bibitem{Pippenger99}
N. Pippenger, ``Entropy and enumeration of Boolean functions,'' {\em IEEE Transactions on Information
Theory}, vol.~45, no.~6, pp.~2096--2100, September 1999. \url{https://doi.org/10.1109/18.782146}

\bibitem{Radhakrishnan97}
J. Radhakrishnan, ``An entropy proof of Bregman's theorem,'' {\em Journal of Combinatorial Theory,
Series~A}, Elsevier Science, vol.~77, no.~1, pp.~161--164, January 1997.
\url{https://doi.org/10.1006/jcta.1996.2727}

\bibitem{Radhakrishnan01}
J. Radhakrishnan, ``Entropy and counting,'' {\em Proceedings of the IIT Kharagpur, Golden Jubilee Volume
on Computational Mathematics, Modelling and Algorithms}, Narosa Publishers, New Delhi, India, pp.~1--25, 2001.
\url{https://www.tcs.tifr.res.in/~jaikumar/Papers/EntropyAndCounting.pdf}

\bibitem{Friedgut04}
E. Friedgut, ``Hypergraphs, entropy, and inequalities,'' {\em The American Mathematical Monthly},
vol.~111, no.~9, pp.~749--760, November 2004. \url{https://doi.org/10.2307/4145187}

\bibitem{GavinskyLSS14}
D. Gavinsky, S. Lovett, M. Saks and S. Srinivasan, ``A tail bound for read-$k$ families of functions,''
{\em Random Structures and Algorithms}, vol.~47, no.~1, pp.~99--108, August 2015.
\url{http://doi.org/10.1002/rsa.20532}

\bibitem{Kahn01}
J. Kahn, ``An entropy approach to the hard-core model on bipartite graphs,'' {\em Combinatorics,
Probability and Computing}, vol.~10, no.~3, pp.~219--237, May 2001.
\url{https://doi.org/10.1017/S0963548301004631}

\bibitem{Kahn02}
J. Kahn, ``Entropy, independent sets and antichains: a new approach to Dedekind's problem,''
{\em Proceedings of the American Mathematical Society}, vol.~130, no.~2, pp.~371--378, June 2001.
\url{https://doi.org/10.1090/S0002-9939-01-06058-0}

\bibitem{MadimanT_IT10}
M. Madiman and P. Tetali, ``Information inequalities for joint distributions,
interpretations and applications," {\em IEEE Transactions on Information Theory},
vol.~56, no.~6, pp.~2699--2713, June 2010.
\url{https://doi.org/10.1109/TIT.2010.2046253}

\bibitem{Sason21}
I. Sason, ``A generalized information-theoretic approach for bounding the number of independent
sets in bipartite graphs,'' {\em Entropy}, vol.~23, no.~3, paper~270, pp.~1--14, 2021.
\url{https://doi.org/10.3390/e23030270}

\bibitem{Sason22}
I. Sason, ``Information inequalities via submodularity, and a problem in extremal graph theory,''
{\em Entropy}, vol.~24, no.~5, paper~597, pp.~1--31, 2022.
\url{https://doi.org/10.3390/e24050597}

\bibitem{Sason_HIM2024}
I. Sason, "Combinatorial applications of the Shearer and Han inequalities in graph theory and Boolean functions,''
{\em Workshop on Information Theory, Boolean Functions, and Lattice Problems}, Hausdorff Research Institute for
Mathematics (HIM), Bonn, Germany, November 18-22, 2024. The recorded talk is at \url{https://www.youtube.com/watch?v=D-poDm7AnLU}

\bibitem{BrightwellW99}
G. Brightwell and P. Winkler, ``Graph homomorphisms and phase transitions,'' {\em Journal
of Combinatorial Theory, Series B}, vol.~77, no.~2, pp.~221--262, 1999.
\url{https://doi.org/10.1006/jctb.1999.1899}

\bibitem{HellN04}
P. Hell and J. Ne\v{s}tril, {\em Graphs and Homomorphisms}, Oxford Lecture Series in Mathematics and
Its Applications, Oxford University Press, 2004.
\url{https://doi.org/10.1093/acprof:oso/9780198528173.001.0001}

\bibitem{Borgs06}
C. Borgs, J. Chayes, L. Lov\'{a}sz, V. T. S\'{o}s, and K. Vesztergombi, ``Counting graph homomorphisms,''
{\em Topics in Discrete Mathematics}, pp.~315--371, Springer, 2006.
\url{https://doi.org/10.1007/3-540-33700-8_18}

\bibitem{CsikvarRS22}
P. Csikv\'{a}ri, N. Ruozzi, and S. Shams, ``Markov random fields, homomorphism counting, and Sidorenko’s
conjecture,'' {\em IEEE Transactions on Information Theory}, vol.~68, no.~9, pp.~6052--6062, September~2022.
\url{https://doi.org/10.1109/TIT.2022.3169487}

\bibitem{FriedgutK98}
E. Friedgut and J. Kahn, ``On the number of copies of one hypergraph in another,''
{\em Israel Journal of Mathematics}, vol.~105, pp.~251--256, 1998.
\url{https://doi.org/10.1007/BF02780332}

\bibitem{Lovasz12}
L. Lov\'{a}sz, {\em Large Networks and Graph Limits}, American Mathematical Society, 2012.
\url{https://doi.org/10.1090/coll/060}

\bibitem{WangTL23}
Z. Wang, J. Tu, and R. Lang, ``Entropy, graph homomorphisms, and dissociation sets,'' {\em Entropy}, vol.~25,
paper.~163, pp.~1--11, January 2023. \url{https://www.mdpi.com/1099-4300/25/1/163}

\bibitem{Sidorenko93}
A. Sidorenko, ``A correlation inequality for bipartite graphs,'' {\em Graphs and Combinatorics}, vol.~9, pp.~201--204,
June 1993. \url{https://doi.org/10.1007/BF02988307}

\bibitem{Simonovits84}
M. Simonovits, ``Extremal graph problems, degenerate extremal problems and super-saturated graphs,''
Progress in graph theory, pp.~419--437, Academic Press, Toronto, ON, 1984.
\url{https://users.renyi.hu/~miki/waterloo.pdf}

\bibitem{Zhao23}
Y. Zhao, {\em Graph Theory and Additive Combinatorics: Exploring Structures and Randomness}, Cambridge University
Press, 2023. \url{https://doi.org/10.1017/9781009310956}

\bibitem{ConlonFS10}
D. Conlon, J. Fox, and B. Sudakov, ``An approximate version of Sidorenko’s conjecture,'' {\em  Geometric
and Functional Analysis}, vol.~20, pp.~1354--1366, October 2010. \url{https://doi.org/10.1007/s00039-010-0097-0}

\bibitem{ConlonFS10b}
D. Conlon, J. Fox, and B. Sudakov, ``Sidorenko's conjecture for a class of graphs: an exposition,'' {\em unpublished note},
2010. \url{https://arxiv.org/abs/1209.0184}

\bibitem{ConlonKLL18}
D. Conlon, J. H. Kim, C. Lee, and J. Lee, ``Some advances on Sidorenko’s conjecture,'' {\em Journal of the London
Mathematical Society}, vol.~98, no.~3, pp.~593--608, December 2018. \url{https://doi.org/10.1112/jlms.12142}

\bibitem{Szegedy15a}
B. Szegedy, ``An information theoretic approach to Sidorenko’s conjecture,'' preprint, January 2015. \url{https://arxiv.org/abs/1406.6738v3}

\bibitem{Szegedy15b}
B. Szegedy, ``Sparse graph limits, entropy maximization and transitive graphs,'' preprint, April 2015. \url{https://arxiv.org/abs/1504.00858v1}

\bibitem{SahSSZ20}
A. Sah, M. Sawhney, D. Stoner, and Y. Zhao, ``A reverse Sidorenko inequality,'' {\em Inventiones Mathematicae}, vol.~221, pp.~665--711,
August 2020. \url{https://doi.org/10.1007/s00222-020-00956-9}

\bibitem{CoverT06}
T. M. Cover and J. A. Thomas, {\em Elements of Information Theory}, second edition,
John Wiley \& Sons, 2006. \url{https://doi.org/10.1002/047174882X}

\bibitem{Han78}
T.~S. Han, ``Nonnegative entropy measures of multivariate symmetric correlations,''
{\em Information and Control}, vol.~36, no.~2, pp.~133--156, February 1978.
\url{https://doi.org/10.1016/S0019-9958(78)90275-9}

\bibitem{Alon81}
N. Alon, ``On the number of subgraphs of prescribed type of graphs with a given number
of edges,'' {\em Israel Journal of Mathematics}, vol.~38, pp.~116-130, 1981.
\url{https://doi.org/10.1007/BF02761855}

\bibitem{GareyJ79}
M. R. Garey and D. S. Johnson, {\em Computers and Intractability: A Guide to the Theory of NP-Completeness},
W. H. Freeman and Company, San Francisco, 1979.

\bibitem{Lovasz79_IT}
L. Lov\'{a}sz, ``On the Shannon capacity of a graph,'' {\em IEEE Transactions on Information Theory},
vol.~25, no.~1, pp.~1--7, January 1979. \url{https://doi.org/10.1109/TIT.1979.1055985}

\bibitem{GrotschelLS81}
M. Gr\"{o}tschel, L. Lov\'{a}sz, and A. Schrijver, ``The ellipsoid method and its
consequences in combinatorial optimization,'' {\em  Combinatorica}, vol.~1, no.~2,
pp.~168--197, June 1981. \url{https://doi.org/10.1007/BF02579273}

\bibitem{Knuth94}
D. E. Knuth, ``The sandwich theorem,'' {\em Electronic Journal of Combinatorics}, vol.~1,
pp.~1--48, 1994. \url{https://doi.org/10.37236/1193}

\bibitem{Lovasz19}
L. Lov\'{a}sz, {\em Graphs and Geometry}, American Mathematical Society, volume~65, 2019.
\url{https://doi.org/10.1090/coll/065}

\bibitem{Sason24}
I. Sason, ``Observations on graph invariants with the Lov\'{a}sz $\vartheta$-function,'' {\em AIMS Mathematics},
vol.~9, no.~6, pp.~15385--15468, April 2024. \url{https://doi.org/10.3934/math.2024747}

\bibitem{BrouwerM22}
A. E. Brouwer and H. Van Maldeghem, {\em Strongly Regular Graphs}, Cambridge University
Press, (Encyclopedia of Mathematics and its Applications, Series Number 182), 2022.
\url{https://doi.org/10.1017/9781009057226}

\bibitem{vanLintW02}
J. H. van Lint and R. M. Wilson, {\em A Course in Combinatorics}, second edition, Cambridge University Press, 2001.
\url{https://doi.org/10.1017/CBO9780511987045}

\bibitem{Sason23}
I. Sason, ``Observations on the Lov\'{a}sz $\vartheta$-function, graph capacity, eigenvalues, and strong products,''
{\em Entropy}, vol.~25, no.~1, paper~104, pp.~1--40, January~2023. \url{https://doi.org/10.3390/e25010104}

\bibitem{Sason25_friendship}
I. Sason, ``On strongly regular graphs and the friendship theorem,'' {\em Mathematics}, vol.~13, no.~6, paper~970,
pp.~1--21, March 2025. \url{https://doi.org/10.3390/math13060970}

\bibitem{Lovasz83}
L. Lov\'{a}sz, {\em Perfect graphs}, In:  Selected Topics in Graph Theory~2, eds. W. L. Beineke, R. J. Wilson,
Academic Press, pp.~55--87, 1983.

\bibitem{Trotignon15}
N. Trotignon, ``Perfect graphs: a survey,'' {\em Topics in Chromatic Graph Theory} (Edited by
L. W. Beineke and R. J. Wilson), Cambridge University Press, Chapter~7, pp.~137--160, 2015.
\url{https://doi.org/10.1017/CBO9781139519793.010}

\end{thebibliography}
\end{document}